\documentclass[10pt]{article} 

\usepackage{latexsym}
\usepackage{amssymb}
\usepackage{amsthm}
\usepackage{amscd}
\usepackage{amsmath}
\usepackage{mathrsfs}
\usepackage{graphicx}
\usepackage{hyperref}
\usepackage{shuffle}
\usepackage{comment}
\usepackage{ytableau}
\usepackage[all]{xy}\input xy \xyoption{frame}\xyoption{dvips}\usepackage[colorinlistoftodos]{todonotes}\usetikzlibrary{chains,scopes,decorations.markings}

\theoremstyle{definition}
\newtheorem* {theorem*}{Theorem}
\newtheorem* {conjecture*}{Conjecture}
\newtheorem{theorem}{Theorem}[section]
\theoremstyle{definition}

\newtheorem* {example*}{Example}

\newtheorem{lemma}[theorem]{Lemma}
\theoremstyle{definition}
\theoremstyle{definition}

\newtheorem{conjecture}[theorem]{Conjecture}\newtheorem{proposition}[theorem]{Proposition}\newtheorem{corollary}[theorem]{Corollary}
\newtheorem *{remark}{Remark}
\theoremstyle{definition}
\newtheorem {example}[theorem]{Example}\theoremstyle{definition}

\theoremstyle{definition}
\theoremstyle{definition}
\theoremstyle{definition}

\numberwithin{equation}{section}

\allowdisplaybreaks[1]
\UseCrayolaColors

\def\modu{\ (\mathrm{mod}\ }

\def\({\left(}
\def\){\right)}

\newcommand{\cC}{\mathcal{C}}

\def\cX{\mathcal{X}}

\def\NN{\mathbb{N}}

\def\ZZ{\mathbb{Z}}

\newcommand\wt{\operatorname{wt}}

\def\barr{\begin{array}}
\def\earr{\end{array}}
\def\ba{\begin{aligned}}
\def\ea{\end{aligned}}
\def\be{\begin{equation}}
\def\ee{\end{equation}}

\def\quand{\quad\text{and}\quad}

\def\hs{\hspace{0.5mm}}

\def\ben{\begin{enumerate}}
\def\een{\end{enumerate}}

\def\hs{\hspace{0.5mm}}

\def\D{\hat D}

\def\e{\textbf{e}}
\newcommand{\xRightarrow}[2][]{\ext@arrow 0359\Rightarrowfill@{#1}{#2}}

\newcommand{\cA}{\mathcal{A}}
\newcommand{\cB}{\mathcal{B}}

\def\arcstart{\ \xy<0cm,-.06cm>\xymatrix@R=.1cm@C=.2cm }\newcommand{\arcstartc}[1]{\ \xy<0cm,-.15cm>\xymatrix@R=.1cm@C=#1cm}

\def\cG{\mathcal{G}}

\def\D{\textsf{D}}
\def\SD{\textsf{SD}}

\usepackage{young}

\def\SD{\mathsf{SD}}

\def\RPP{{\rm RPP}}
\def\ShYTP{{\rm ShYT}_P}
\def\ShYTQ{{\rm ShYT}_Q}

\def\MRPP{{\rm ShRPP}}

\def\SetShYTP{{\rm SetShYT}_P}
\def\SetShYTQ{{\rm SetShYT}_Q}

\def\ValShYTP{{\rm ShBT}_P}
\def\ValShYTQ{{\rm ShBT}_Q}

\def\col{\mathsf{cols}}
\def\overlap{\mathsf{overlap}}

\def\HQ{H\hspace{-0.2mm}Q}
\def\GQ{G\hspace{-0.2mm}Q}
\def\GP{G\hspace{-0.2mm}P}

\def\gq{g\hspace{-0.2mm}q}
\def\gp{g\hspace{-0.2mm}p}

\def\jq{j\hspace{-0.2mm}q}
\def\jp{j\hspace{-0.2mm}p}

\def\JQ{J\hspace{-0.2mm}Q}
\def\JP{J\hspace{-0.2mm}P}

\def\Rem{\mathsf{Rem}}
\def\ss{/\hspace{-1mm}/}

\def\unprime{\mathsf{unprime}}
\def\unprimemax{\unprime_{\max}}

\begin{document}
\title{Expanding $K$-theoretic Schur $Q$-functions}

\author{
Yu-Cheng Chiu\thanks{
Department of Mathematics, ETH Z\"{u}rich, \tt chiuyu@student.ethz.ch
}
\and
Eric Marberg\thanks{
Department of Mathematics, HKUST, \tt emarberg@ust.hk
}
}

\date{}

\maketitle

\begin{abstract}
We derive several identities involving Ikeda and Naruse's $K$-theoretic Schur $P$- and $Q$-functions. Our main result is a formula conjectured by Lewis and the second author which expands each $K$-theoretic Schur $Q$-function in terms of $K$-theoretic Schur $P$-functions. This formula extends to some more general identities relating the skew and dual versions of both power series. We also prove a shifted version of Yeliussizov's skew Cauchy identity for symmetric Grothendieck polynomials. Finally, we discuss some conjectural formulas for the dual $K$-theoretic Schur $P$- and $Q$-functions of Nakagawa and Naruse. We show that one such formula would imply a basis property expected of the $K$-theoretic Schur $Q$-functions.
\end{abstract}

\setcounter{tocdepth}{2}

\section{Introduction} 
This article proves some identities relating the \emph{$K$-theoretic Schur $P$- and $Q$-functions}
introduced by Ikeda and Naruse in \cite{IkedaNaruse}.
To motivate the definition of these power series and to frame our main results, we start by reviewing some classical background material
on generating functions for shifted tableaux.

Let $\lambda = (\lambda_1>\lambda_2>\dots>0)$ be a \emph{strict partition}, that is,
a strictly decreasing sequence of positive integers.
The \emph{shifted diagram} of $\lambda $ is the set of pairs
$\SD_{\lambda} := \{ (i,j) \in \ZZ \times \ZZ : 0 < i \leq  j< i + \lambda_i\}$.
We usually refer to the elements of this set as ``positions'' or ``boxes.''

A \emph{shifted tableau} of shape $\lambda$ is a filling 
of $\SD_{\lambda}$ by positive half-integers. 
For any $i \in \ZZ$ let $i' :=i-\frac{1}{2}$. Then one may think of the entries 
of a shifted tableau as consisting of positive integers $i$ and primed numbers $i'$.
A shifted tableau is \emph{semistandard} if the following conditions hold:
\ben
\item[(S1)] The entries in each row and column are weakly increasing.
\item[(S2)] No unprimed number $i$ occurs more than once in a given column.
\item[(S3)] No primed number $i'$ occurs more than once in a given row.
\een
Let $\ShYTQ(\lambda)$ denote the set of semistandard shifted tableaux of shape $\lambda$. Define $\ShYTP(\lambda) \subseteq \ShYTQ(\lambda)$
to be the subset of tableaux also satisfying:
\ben
\item [(S4)] No primed number occurs in any diagonal position $(j,j)\in \SD_{\lambda}$.
\een
We refer to elements of $\ShYTP(\lambda)$  and $\ShYTQ(\lambda)$ 
as \emph{$P$-shifted} and \emph{$Q$-shifted tableaux}, respectively.
We draw shifted tableaux in French notation:
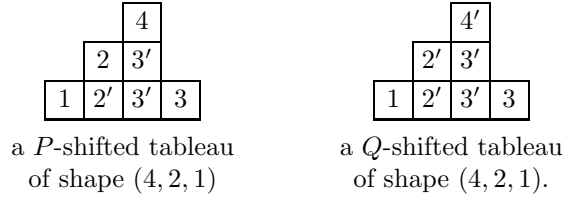
\[
\ytableausetup{boxsize=0.5cm,aligntableaux=center}
\barr{c} \begin{ytableau}
\none & \none & 4 & \none\\
\none &  2  & 3' & \none\\
1 & 2' & 3' & 3
\end{ytableau}
\\[-7pt]\\
\text{a $P$-shifted tableau}\\
\text{of shape $(4,2,1)$}
\earr
\quad\quad\quad
\barr{c} 
\begin{ytableau}
\none & \none & 4' & \none\\
\none &  2'  & 3' & \none\\
1 & 2' & 3' & 3
\end{ytableau}
\\[-7pt]\\
\text{a $Q$-shifted tableau}\\
\text{of shape $(4,2,1)$}.
\earr
\]

The \emph{weight} of a shifted tableau $T$ is the monomial $x^T := \prod_{i\geq 1} x_i^{m_i}$
where $m_i$ is the number of times  that $i$ or $i'$ appears in $T$.
For example, we have $x^T = x_1 x_2^2 x_3^3x_4$ for both of the shifted tableaux shown as examples above.
%
The \emph{Schur $P$- and $Q$-functions} indexed by $\lambda$ are the power series  
\be
P_{\lambda} := \sum_{T\in\ShYTP(\lambda)} x^T\quand
Q_{\lambda} := \sum_{T\in\ShYTQ(\lambda)} x^T.
\ee
Any way of toggling the primes in the diagonal entries of a $Q$-shifted tableau 
results in another $Q$-shifted tableau of the same weight, so it is clear that 
\be\label{easy-eq}
Q_{\lambda} = 2^{\ell(\lambda)}  P_{\lambda}\quad\text{where }\ell(\lambda) := |\{ i : \lambda_i>0\}| = | \{ i : (i,i) \in \SD_\lambda\}|.
\ee
It is well-known that $P_\lambda$ and $Q_\lambda$ are symmetric functions of bounded degree. They were first defined in work of Schur on the projective representations of the symmetric group
but have since appeared in various other contexts.

We are interested in generalizations of $P_\lambda$ and $Q_\lambda$ that  are similar generating functions for set-valued shifted tableaux. 
A \emph{set-valued shifted tableau} of shape $\lambda$
is a filling of $\SD_{\lambda}$ by nonempty finite subsets of $ \{ \frac{1}{2} i : 0< i \in \ZZ\} = \{ 1' < 1 < 2'<2<\dots\}$.
We consider a sequence of such subsets $S_1,S_2,S_3,\dots$ to be \emph{weakly increasing}
if $\max(S_i) \leq \min(S_{i+1})$ for all $i$.
With this convention,
we may define a set-valued shifted tableau to be \emph{semistandard} if it satisfies the same conditions (S1)-(S3) as above.

We write $\SetShYTQ(\lambda)$ 
for the set of all semistandard set-valued shifted tableaux of shape $\lambda$,
and $\SetShYTP(\lambda)$ for the subset of such tableaux also
satisfying (S4).
We refer to elements of $\SetShYTP(\lambda)$  and $\SetShYTQ(\lambda)$ 
as \emph{set-valued $P$-shifted} and \emph{set-valued $Q$-shifted tableaux}, respectively:
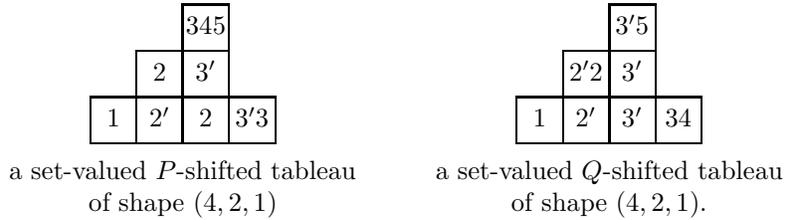
\[
\ytableausetup{boxsize=0.6cm,aligntableaux=center}
\barr{c} \begin{ytableau}
\none & \none & 345 & \none\\
\none &  2  & 3' & \none\\
1 & 2' & 2 & 3'3
\end{ytableau}
\\[-7pt]\\
\text{a set-valued $P$-shifted tableau}\\
\text{of shape $(4,2,1)$}
\earr
\quad\quad
\barr{c} \begin{ytableau}
\none & \none & 3'5 & \none\\
\none &  2' 2  & 3' & \none\\
1 & 2' & 3' & 34
\end{ytableau}
\\[-7pt]\\
\text{a set-valued $Q$-shifted tableau}\\
\text{of shape $(4,2,1)$.}
\earr
\]

The \emph{weight} $x^T$ of a set-valued shifted tableau $T$ is defined in the same way
as in the non-set-valued case; for both tableaux in the preceding example
one has  $x^T = x_1 x_2^3 x_3^4 x_4x_5$.
Write   $T_{ij}$ for the entry of a set-valued shifted tableau in position $(i,j)$
and define 
$|T| := \sum_{(i,j) \in \SD_{\lambda}} |T_{ij}|$ and $ |\lambda| := |\SD_{\lambda}|.$
Then $|T|-|\lambda|$ is the difference between the degree of $x^T$ and the size of $\SD_{\lambda}$.
Finally, let $\beta$ be a variable that commutes with each $x_i$.
The \emph{$K$-theoretic Schur $P$- and $Q$-functions} indexed by $\lambda$ are
the power series in $\ZZ[\beta][[x_1,x_2,\dots]]$ given by
\be\label{GP-GQ-def}
\ba
\GP_{\lambda}  &:= \sum_{T\in\SetShYTP(\lambda)} \beta^{|T|-|\lambda|} x^T,\\
\GQ_{\lambda}  &:= \sum_{T\in\SetShYTQ(\lambda)} \beta^{|T|-|\lambda|} x^T.
\ea\ee
We recover
$P_{\lambda}$ from $\GP_{\lambda}$
and
$Q_{\lambda}$ from $\GQ_{\lambda}$ by setting $\beta=0$.
Both $\GP_\lambda$ and $\GQ_\lambda$ are symmetric in the $x_i$ variables \cite[\S3.4]{IkedaNaruse} and homogeneous of degree $|\lambda| $
if we set $\deg(\beta) = -1$ and $\deg(x_i)=1$.

Ikeda and Naruse   introduced   these functions in  \cite{IkedaNaruse}
for applications in $K$-theory.
Specializations of $\GP_\lambda$ and $\GQ_\lambda$ represent the structure sheaves of Schubert varieties in the $K$-theory of the maximal isotropic Grassmannians of orthogonal and symplectic types \cite[Cor. 8.1]{IkedaNaruse}. 
 More precisely, the $\GP$- and $\GQ$-functions represent Schubert classes in connective $K$-theory,
 so can be turned into cohomology classes or elements of the Grothendieck ring of vector bundles
 on setting $\beta=0$ and $\beta=1$, respectively.
The $\GP$- and $\GQ$-functions are also ``stable limits'' of connective $K$-theory classes of
 orbit closures for symplectic and orthogonal groups acting on the type A flag variety \cite{MP2020,MP2021}.
For more results about these functions and various extensions, see \cite{NakagawaNaruse,NakagawaNaruse0,Naruse}.

Our first main result is a $K$-theoretic analogue of equation \eqref{easy-eq}.
The relevant identity is subtler than in the classical case,
and was predicted as \cite[Conj. 5.15]{LM2021}.
It expresses each $K$-theoretic Schur $Q$-function as a finite linear combination of
$K$-theoretic Schur $P$-functions with integer coefficients. 
 
If $\lambda = (\lambda_1>\lambda_2>\dots>0)$ and $\mu = (\mu_1>\mu_2>\dots>0)$ are strict partitions with $\mu_i \leq \lambda_i$ for all $i$
then we write  $\mu \subseteq \lambda$ and define $\SD_{\lambda/\mu} := \SD_{\lambda}\setminus \SD_{\mu}$ and 
$|\lambda/\mu| := |\SD_{\lambda/\mu} |$.
We also let 
$\col(\lambda/\mu) := | \{ j : (i,j) \in \SD_{\lambda/\mu} \text{ for some }i\}|$
denote the number of distinct columns occupied by the positions in $\SD_{\lambda/\mu}$.
%


\begin{theorem}
\label{to-prove}
If $\mu$ is a strict partition with $\ell(\mu)$ parts then
\be\label{q-to-p-eq}
\GQ_\mu = 2^{\ell(\mu)}  \sum_{\lambda}  (-1)^{\col(\lambda/\mu)}  (-\beta/2)^{|\lambda/\mu| }  \GP_\lambda
\ee
where the sum is over  strict partitions $\lambda\supseteq \mu$ 
with $\ell(\lambda) = \ell(\mu)$ such that $\SD_{\lambda/\mu}$ is a \emph{vertical strip},
that is,   a subset with at most one position in each row. 
\end{theorem}

For example, it holds that 
$\GQ_{(3,2)} = 4  \GP_{(3,2)} + 2\beta  \GP_{(4,2)} - \beta^2 \GP_{(4,3)}$
and 
$ \GQ_{(n)} = 2  \GP_{(n)} + \beta  \GP_{(n+1)}$
for all integers $n>0$.
This formula applies even when 
$\mu = \emptyset$ is the empty partition, as then the sum has only one term 
indexed by $\lambda = \emptyset$, giving $\GQ_\emptyset = \GP_\emptyset = 1$.
We prove Theorem~\ref{to-prove} in Section~\ref{res-sect}.

The additional complexity in \eqref{q-to-p-eq} compared to \eqref{easy-eq} is related to the fact that
in a set-valued $Q$-shifted tableau,
a diagonal entry may contain both $i$ and $i'$.
When this happens there is no simple way to remove all primes from the diagonal without changing the relevant weight.

As a corollary, we may classify when only positive coefficients appear in the $\GP$-expansion of $\GQ_\mu$.
Let $\NN := \{0,1,2,\dots,\}$.

\begin{corollary}\label{positive_case}
If $\mu$ is a strict partition then
$\GQ_\mu$ is an $\NN[\beta]$-linear combination of $\GP$-functions
if and only if all distinct parts of $\mu$ differ by at least two. 
\end{corollary}

\begin{proof}
It suffices by Theorem~\ref{to-prove} to observe that there exists a strict partition $\lambda \supseteq \mu$ with $\ell(\lambda) =\ell(\mu)$ such that $\SD_{\lambda/\mu}$
 is a vertical strip and $\col(\lambda/\mu) \not\equiv |\lambda/\mu| \modu 2)$ if and only if 
 $\mu_i - \mu_{i+1} = 1$ for some $i \in [\ell(\mu)-1]$.
\end{proof}

We also prove a few more results.
Theorem~\ref{to-prove} has
 some enumerative consequences which we discuss in Section~\ref{bij-sect}.
 
The $\GP$- and $\GQ$-functions have skew versions $\GP_{\lambda/\mu}$ and $\GQ_{\lambda/\mu}$,
which are generating functions for set-valued tableaux of shifted skew shapes.
In Section~\ref{skew-sect} we derive an extension of Theorem~\ref{to-prove}
for these power series, along with some other related identities.

There are also dual power series $\gp_\lambda$ and $\gq_\lambda$ defined by Nakagawa and Naruse \cite{NakagawaNaruse}
from $\GP_\lambda$ and $\GQ_\lambda$
via a Cauchy identity.
Section~\ref{dual-sect} contains some further results about these functions,
including a dual form of Theorem~\ref{to-prove} (see Corollary~\ref{to-prove2}) and a skew Cauchy identity (see Theorem~\ref{GP-cauchy-thm}).

In Section~\ref{last-sect}, we recall a conjectural formula for $\gp_\lambda$ and $\gq_\lambda$ from \cite{NakagawaNaruse}.
We then explain a new conjectural formula for the related functions $\jp_\lambda := \omega(\gp_\lambda)$ and $\jq_\lambda := \omega(\gq_\lambda)$ obtained by applying the algebra automorphism $\omega $ that sends $ s_\lambda \mapsto s_{\lambda^\top}$.
We show that these new conjectures would imply a  conjecture of Ikeda and Naruse 
about the $\GQ$-functions forming a $\ZZ[\beta]$-basis for a ring.

\subsection*{Acknowledgements}

This work was partially supported by grants ECS 26305218 and GRF 16306120
from the Hong Kong Research Grants Council.
We thank Joel Lewis for several helpful comments.

\section{Preliminaries}

%
%

Fix a positive integer $n$
and continue to let $\beta, x_1,x_2,\dots $ be commuting variables. 
For any $f \in \ZZ[\beta][[x_1,x_2,\dots]]$ let 
$ f(x_1,x_2,\dots,x_n) \in \ZZ[\beta][x_1,x_2,\dots,x_n]$
be the polynomial obtained by setting $x_{n+1} = x_{n+2} = \dots=0$.
Also define
\be x\oplus y := x +y + \beta xy\quand x\ominus y :=\tfrac{x-y}{1+\beta y} .\ee
If  $\lambda=(\lambda_1,\lambda_2,\dots)$ is a finite sequence of integers then let $x^\lambda := \prod_{i} x_i^{\lambda_i}$.

Ikeda and Naruse use the following formulas
as their definition of the $K$-theoretic Schur $P$- and $Q$-functions \cite[Def. 2.1]{IkedaNaruse}.
They derive the set-valued
tableau generating functions given in the introduction as \cite[Thm. 9.1]{IkedaNaruse}.

\begin{theorem}[{See \cite{IkedaNaruse}}]
\label{ik-thm}
If $\lambda$ is a strict partition with $r := \ell(\lambda) \leq n$  then
\[\ba
\GP_\lambda(x_1,x_2,\dots,x_n) &= \frac{1}{(n-r)!} \sum_{w \in S_n} w\( x^\lambda \prod_{i=1}^r \prod_{j=i+1}^n \frac{x_i\oplus x_j}{x_i\ominus x_j}\),
\\
\GQ_\lambda(x_1,x_2,\dots,x_n) &= \frac{1}{(n-r)!} \sum_{w \in S_n} w\( x^\lambda \prod_{i=1}^r (2+\beta x_i) \prod_{j=i+1}^n \frac{x_i\oplus x_j}{x_i\ominus x_j}\),
\ea\] 
where  $w \in S_n$ acts on rational functions by permuting the $x_i$ variables while fixing $\beta$.
\end{theorem}

For classical background on the Schur $P$ and $Q$-functions, see  \cite[\S{III.8}]{Macdonald},
\cite[\S5-\S9]{Stembridge1989}, or the appendix in \cite{Stembridge1997}.
As $\lambda$ ranges over all strict partitions the functions $P_\lambda$ (respectively, $Q_\lambda$)
are a $\ZZ$-basis for a subring of $\ZZ[[x_1,x_2,\dots]]$.
The same is true of the polynomials $P_\lambda(x_1,x_2,\dots,x_n)$
(respectively, $Q_\lambda(x_1,x_2,\dots,x_n)$)
if $\lambda$ ranges over strict partitions with $\ell(\lambda)\leq n$ \cite[\S{III.8}]{Macdonald}.

Since $P_\lambda=\GP_\lambda|_{\beta=0}$ and $Q_\lambda=\GQ_\lambda|_{\beta=0}$
it follows that the $\GP_\lambda$'s (respectively, the $\GQ_\lambda$'s) are linearly independent over $\ZZ[\beta]$,
as are the polynomials $\GP_\lambda(x_1,x_2,\dots,x_n)$
(respectively, $\GQ_\lambda(x_1,x_2,\dots,x_n)$) as $\lambda$ ranges over all strict partitions with at most $n$ parts. 

In fact, Ikeda and Naruse show that the sets
$\{ \GP_\lambda(x_1,x_2,\dots,x_n) : \ell(\lambda) \leq n\}$
and
$\{ \GQ_\lambda(x_1,x_2,\dots,x_n) : \ell(\lambda) \leq n\}$ are both $\ZZ[\beta]$-bases 
for subrings of $\ZZ[\beta][x_1,x_2,\dots,x_n]$ \cite[Thm. 3.1 and Prop. 3.2]{IkedaNaruse}.
An analogous basis property is known to
hold for the set of all formal power series $\GP_\lambda$'s
and is expected to hold for the $\GQ_\lambda$'s; see the discussion before Corollary~\ref{last-cor}. 

The $\GP$- and $\GQ$-power series are generalizations of Ivanov's \emph{factorial $P$- and $Q$-functions}  \cite{Ivanov}, which have another generalization studied by Okada  in \cite{Okada}. Comparing Theorem~\ref{ik-thm} and \cite[Lem. 2.4]{Okada} suggests a common generalization of these functions which might be interesting to consider in future work.

\section{Expansions}\label{res-sect}
We prove Theorem~\ref{to-prove} in this section.
For $m \in \NN$ let $[m] := \{1,2,\dots,m\}$.
Fix an integer $n>0$. Given a nonzero vector $\lambda \in \NN^n$, we define  rational functions
\be\label{AB-eq}
\ba
A_\lambda  &:=\frac{1}{r!}\sum_{w\in S_r} w\( x^\lambda \prod_{i=1}^r \prod_{j=i+1}^n \frac{x_i\oplus x_j}{x_i\ominus x_j}\),
\\
B_\lambda  &:= \frac{1}{r!}\sum_{w\in S_r} w\( x^\lambda \prod_{i=1}^r (2+\beta x_i) \prod_{j=i+1}^n \frac{x_i\oplus x_j}{x_i\ominus x_j}\),
\ea\ee
where $r := \max\{ i \in [n] : \lambda_i \neq 0\}$.
Note the implicit dependence on $n$ in these formulas.
For convenience we also set $A_0 = B_0 = 1$.

\begin{lemma}\label{key-lem1}
Let $\lambda$ be a strict partition with $r:= \ell(\lambda) \leq n$.
Fix $m \in [r]$. Let 
\[\mu := (\lambda_1 > \lambda_2 > \dots > \lambda_m)
\quand
\nu:= (\lambda_{m+1} > \lambda_{m+2} > \dots > \lambda_r).\]
Then it holds that
\[
\ba
\GP_\lambda(x_1,x_2,\dots,x_n)  &= \frac{1}{(n-m)!} \sum_{w\in S_n}  w\( A_\mu \GP_\nu(x_{m+1},x_{m+2},\dots,x_n)  \),
\\
\GQ_\lambda(x_1,x_2,\dots,x_n)  &= \frac{1}{(n-m)!} \sum_{w\in S_n}  w\( B_\mu \GQ_\nu(x_{m+1},x_{m+2},\dots,x_n)  \).
\ea
\]

\end{lemma}

\begin{proof}
Choose any polynomial  $f(x) \in \ZZ[\beta][x]$ and let $Z_i := f(x_i) \prod_{j=i+1}^n \frac{x_i \oplus x_j}{x_i \ominus x_j}$
for each $i \in [r]$. By Theorem~\ref{ik-thm}, the expression
\be\label{first-expr}
\frac{1}{(n-r)!}\sum_{w\in S_n} w \( x^{\lambda} Z_1Z_2\cdots Z_r\)
\ee
 gives $\GP_\lambda(x_1,\dots,x_n) $ when $f(x) = 1$
and  $\GQ_\lambda(x_1,\dots,x_n) $ when $f(x)=2+\beta x$. 
Let $ S_{m}$ and $H_{n-m} \cong S_{n-m}$  be the subgroups of permutations in $S_n$ fixing each $i \in[n]\setminus[m]$
and $i \in [m]$ respectively.
Then we can rewrite \eqref{first-expr}
as
\[
\tfrac{1}{(n-r)! m! (n-m)!} \sum_{w\in S_n} w \( \sum_{(g,h)\in S_m\times H_{n-m}} gh\Bigl( x^\mu Z_1\cdots Z_m\Bigr) gh\Bigl(x^{\tilde \nu} Z_{m+1}\cdots Z_r\Bigr)\)
\]
where $\tilde \nu$ is the sequence formed from $\nu$ by prepending $m$ zeros.
The subgroups $S_m$ and $H_{n-m}$ commute, and 
each $h \in H_{n-m}$ fixes $x^\mu Z_1\cdots Z_m$ while each $g \in S_m$ fixes $x^{\tilde \nu} Z_{m+1}\cdots Z_r$.
The preceding expression is therefore equal to 
\[
\tfrac{1}{(n-r)! } \sum_{w\in S_n} w \(\tfrac{1}{m!}\sum_{g \in S_m} g\Bigl( x^\mu Z_1\cdots Z_m\Bigr)\cdot   \tfrac{1}{(n-m)!}\sum_{h \in H_{n-m}} h\Bigl(x^{\tilde \nu} Z_{m+1}\cdots Z_r\Bigr)\).
\]
If $f(x)=1$ then the internal sums here are
 $
 \tfrac{1}{m!}\sum_{g \in S_m} g\Bigl( x^\mu Z_1\cdots Z_m\Bigr) = A_\mu$
 and 
 $
  \tfrac{1}{(n-m)!}\sum_{h \in H_{n-m}} h\Bigl(x^{\tilde \nu} Z_{m+1}\cdots Z_r\Bigr)= \GP_\nu(x_{m+1},\dots,x_n)$ by Theorem~\ref{ik-thm}.
This proves the first identity. The other follows by taking $f(x) = 2+\beta x$.
\end{proof}

For $r \in [n]$ let $\Pi^{r,n}(x) := \prod_{i=1}^r \prod_{j=i+1}^n \frac{x_i\oplus x_j}{x_i\ominus x_j}$.
\begin{lemma}\label{lemma1}
Choose integers $1\leq p<q \leq r$ and let $t_{pq} := (p,q) \in S_r$. Then 
\[ t_{pq} \((1+\beta x_p)^{q-p}  \Pi^{r,n}(x) \) = - (1+\beta x_p)^{q-p}\Pi^{r,n}(x).\]
Consequently if  $f(x) \in \ZZ[\beta][x_1,x_2,\dots]$
is
any polynomial  fixed by $t_{pq}$
then
\[\sum_{w \in S_r} w\(f(x)  (1+\beta x_p)^{q-p}  \Pi^{r,n}(x) \) = 0.\]
\end{lemma}

\begin{proof}
The first identity is a straightforward exercise in algebra.
The second claim follows since $\sum_{w \in S_r} w = \frac{1}{2} \sum_{w \in S_r} w (1+ t_{pq}) \in \ZZ S_r$.
\end{proof}

Choose an integer $m \in [n]$ and let 
$\delta :=(m,m-1,\dots,3,2,1).$
For each $i \in [m]$ let $\e_i := (0,\dots,0,1,0,\dots,0)$ be the standard basis vector in $\ZZ^m$.

\begin{lemma}\label{cancel}
Let $v \in \NN^m$.
Then  $A_{\delta+v} + \beta A_{\delta + v+\e_i} = 0$
whenever  $v_{i+1} = v_i + 1$ for some $ i\in [m-1]$ 
or
$v_{i+2} =v_{i+1}+1=v_i+1$ for some $ i\in [m-2]$.
\end{lemma}

\begin{proof}
First assume $v_{i+1} = v_i + 1$ for some $ i\in [m-1]$.
Then  $x^{\delta+v}$ is a polynomial fixed by $t_{i,i+1}$ so
$
A_{\delta+v} + \beta A_{\delta + v+\e_i} 
= \frac{1}{m!} \sum_{w\in S_m} w ( x^{\delta+v} (1+\beta x_{i})  \Pi^{m,n}(x) )
=0
$ by Lemma~\ref{lemma1}.

Next suppose that $v_{i+2} =v_{i+1}+1=v_i+1$ for some $ i\in [m-2]$.
Let $\alpha := \delta+v-\e_i$. 
Then we can write $\beta (A_{\delta+v} + \beta A_{\delta + v+\e_i})$ as
\[ \tfrac{1}{m!} \sum_{w\in S_m} w ( x^{\alpha} (1+\beta x_i)^2   \Pi^{m,n}(x) )
-
  \tfrac{1}{m!} \sum_{w\in S_m} w ( x^{\alpha}   (1+\beta x_i)  \Pi^{m,n}(x) ).\]
Since  $x^{\alpha}$ is  fixed by $t_{i,i+2}$ and $t_{i,i+1}$,
applying Lemma~\ref{lemma1} with $(p,q) = (i,i+2)$ and $(p,q)=(i,i+1)$
shows that both terms are zero.
The ring of rational functions in $\beta,x_1,x_2,\dots,x_n$ is an integral domain, so
$
A_{\delta+v} + \beta A_{\delta + v+\e_i} =0
$.
\end{proof}

After expanding $B_\delta$ in terms of the $A_\lambda$'s, one can apply many cancellations from Lemma~\ref{cancel}.
We will use the next lemma to organize these cancellations. This lemma involves a certain directed graph $\cG_m$ 
for $m\geq 2$ which we define inductively.
In general, the vertex set of $\cG_m$ consists of 
all nonempty subsets of $[m]$ excluding sets of the form $[i]$ for $i$ odd and 
including two copies of  $[i]$ for $i$ even.
When $m\in \{2,3\}$  the graph $\cG_m$ is given explicitly by
\[
\cG_2:=\boxed{\begin{tikzpicture}[baseline=(z.base)]  \pgfsetlinewidth{1bp}
\node (z) at (0, 0.5) {};
\node (a) at (0,0) {$\{1,2\}$};
\node (b) at (1,0) {$\{1,2\}$};
\node (c) at (0.5,1) {$\{2\}$};
\draw [->] (a) -- (c);
\draw [->] (b) -- (c);
\end{tikzpicture}}
 \quand
\cG_3:=\boxed{\begin{tikzpicture}[baseline=(z.base)]  \pgfsetlinewidth{1bp}
\node (z) at (0, 0.5) {};
\node (a) at (0,0) {$\{1,2\}$};
\node (b) at (1,0) {$\{1,2\}$};
\node (c) at (2,0) {$\{1,3\}$};
\node (d) at (3,0) {$\{2,3\}$};
\node (e) at (0.5,1) {$\{2\}$};
\node (f) at (2.5,1) {$\{3\}$};
\draw [->] (a) -- (e);
\draw [->] (b) -- (e);
\draw [->] (c) -- (f);
\draw [->] (d) -- (f);
\end{tikzpicture}}.
\]
Assume that $m\geq 4$ and that $\cG_{m-2}$ and $\cG_{m-1}$ have been constructed.
Let $\cA$ be the set of vertices $S \in \cG_m$ with $\{m-2,m-1,m\}\subseteq S$,
let $\cB$ be the set of vertices $S \in \cG_m$ with $m \notin S$,
and let $\cC$ be the set of remaining vertices in $\cG_m$.
All of the doubled vertices $[i] \in \cG_m$ for $i =2,4,6,\dots$ belong to $\cA$ or $\cB$.

The edges of $\cG_m$ are given as follows. The three sets $\cA$, $\cB$, and $\cC$ are each unions of connected components.
An edge goes from $S \in \cA$ to $T \in \cA$ if and only if the edge $S \setminus\{m-1,m\} \to T \setminus\{m-1,m\}$ exists in $\cG_{m-2}$.
An edge goes from $S \in \cB$ to $T \in \cB$ if and only if the same edge $S \to T$ exists in $\cG_{m-1}$.
The elements of $\cC$ consist of the distinct unions
$ S \sqcup \{m-1,m\}, $ $ S \sqcup \{m-2,m\},$ and $ S \sqcup \{m\} $ as $S$ ranges over all subsets of $[m-3]$,
and for each $S\subseteq [m-3]$ there are edges $S \sqcup \{m-1,m\} \to S \sqcup \{m\}$
and $S \sqcup \{m-2,m\} \to S \sqcup \{m\}$.

\begin{example} If $m=4$ then
\ben
\item[] $\cA$ has elements $\{1,2,3,4\}, \{1,2,3,4\}, \{2,3,4\}$; 
\item[] $\cB$ has elements $\{1,2\}, \{1,2\}, \{3\}, \{1,3\}, \{2,3\}, \{2\}$; 
\item[] $\cC$ has elements $\{1,3,4\}, \{1,2,4\}, \{1,4\}, \{3,4\}, \{2,4\}, \{4\}$;
\een
and the graph $\cG_4$ is
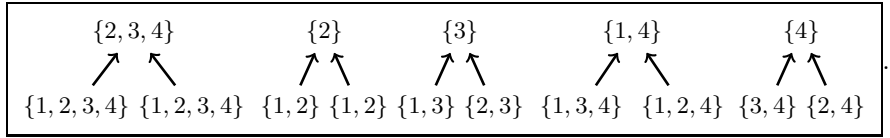
\[
{\small\boxed{\begin{tikzpicture}[baseline=(z.base),xscale=0.85]  \pgfsetlinewidth{1bp}
\node (z) at (0, 0.5) {};
\node (a) at (0,0) {$\{1,2,3,4\}$};
\node (b) at (1.8,0) {$\{1,2,3,4\}$};
\node (c) at (0.9,1) {$\{2,3,4\}$};
\draw [->] (a) -- (c);
\draw [->] (b) -- (c);
\end{tikzpicture}
\begin{tikzpicture}[baseline=(z.base),xscale=0.9]  \pgfsetlinewidth{1bp}
\node (z) at (0, 0.5) {};
\node (a) at (0,0) {$\{1,2\}$};
\node (b) at (1,0) {$\{1,2\}$};
\node (c) at (2,0) {$\{1,3\}$};
\node (d) at (3,0) {$\{2,3\}$};
\node (e) at (0.5,1) {$\{2\}$};
\node (f) at (2.5,1) {$\{3\}$};
\draw [->] (a) -- (e);
\draw [->] (b) -- (e);
\draw [->] (c) -- (f);
\draw [->] (d) -- (f);
\end{tikzpicture}
\begin{tikzpicture}[baseline=(z.base),xscale=0.9]  \pgfsetlinewidth{1bp}
\node (z) at (0, 0.5) {};
\node (a) at (0,0) {$\{1,3,4\}$};
\node (b) at (1.5,0) {$\{1,2,4\}$};
\node (c) at (2.75,0) {$\{3,4\}$};
\node (d) at (3.75,0) {$\{2,4\}$};
\node (e) at (0.75,1) {$\{1,4\}$};
\node (f) at (3.25,1) {$\{4\}$};
\draw [->] (a) -- (e);
\draw [->] (b) -- (e);
\draw [->] (c) -- (f);
\draw [->] (d) -- (f);
\end{tikzpicture}
}}.
\]
\end{example}

\begin{lemma}\label{graph}
For each $m\geq 2$ the graph $\cG_m$ has the following properties:
\ben
\item[(a)] Each directed edge has the form $S \sqcup \{i\} \to S$ for an integer $i$ with either $S \cap \{i,i+1\} = \{i+1\}$
or  $S \cap \{i,i+1,i+2\} = \{i+2\}$.

\item[(b)] Each vertex either has indegree $2$ and outdegree $0$ or
has indegree $0$ and outdegree $1$.
\een
\end{lemma}

\begin{proof}
The explicit graphs $\cG_2$ and $\cG_3$ have these properties. Assume $m\geq 4$ and 
define the vertex subsets $\cA$, $\cB$, and $\cC$ in $\cG_m$ as above.
The edges in $\cC$ have properties (a) and (b) by definition.
The edges in $\cB$ have properties (a) and (b) by induction,
since these vertices form a copy of $\cG_{m-1}$.
The edges in $\cA$ also have the desired properties by induction,
since the subgraph on these vertices is isomorphic to $\cG_{m-2}$ via the map $S \mapsto S \setminus\{m-1,m\}$.
\end{proof}

Our last step before proving Theorem~\ref{to-prove} is to derive a simplified form of the desired 
identity involving the functions $A_\lambda$ and $B_\lambda$.

\begin{lemma}\label{key-lem2}
Suppose $\mu = (q,q-1,q-2,\dots,p)$ for integers $q\geq p > 0$.
Then 
\be\label{kl2-eq} B_\mu = 2^{\ell(\mu)} \sum_{\lambda}  (-1)^{\col(\lambda/\mu)}  (-\beta/2)^{|\lambda/\mu|}  A_\lambda\ee
where the sum is over strict partitions $\lambda\supseteq \mu$ 
with $\ell(\lambda) = \ell(\mu)$ such that $\SD_{\lambda/\mu}$ is a vertical strip,
and $\col(\lambda/\mu)$ is the number of  columns occupied by  $\SD_{\lambda/\mu}$.
\end{lemma}

\begin{proof}
We first prove the lemma in the case when $q=m$ and $p=1$. Then $\mu=\delta:=(m,m-1,\dots,3,2,1)$ and \eqref{kl2-eq} becomes
\be\label{key-result-eq}
B_\delta  = 2^m A_\delta - 2^m\sum_{i=1}^m (-\beta/2)^i A_{\delta + \e_1 + \e_2 + \dots + \e_i}.
\ee
For a subset $S \subseteq [m]$, let $\e_S := \sum_{i \in S} \e_i$.
It follows by expanding the definition of $B_\delta$ in \eqref{AB-eq} that \eqref{key-result-eq} is equivalent to 
$
\sum_{  S\subseteq [m]} 2^{m-|S|}\beta^{|S|}A_{\delta+\e_S} = 2^m A_\delta - \sum_{i=1}^m (-1)^i 2^{m-i}\beta^i A_{\delta + \e_{[i]}}
$, 
which we can rewrite as the identity
\be\label{staircase_toprove}
\sum_{ S\subseteq [m]} \chi(S)2^{m-|S|}\beta^{|S|}A_{\delta+\e_S}=0
\ee
where $\chi(S)$  is defined to be $2$ if $S=[i]$ for any $i \in \{2,4,6,\dots\}$,
$0$ if $S= \varnothing$ or $S = [i]$ for any $i \in \{1,3,5,\dots\}$, and $1$ otherwise.
By Lemma~\ref{graph}, the left-hand side of \eqref{staircase_toprove} is precisely 
\[
\sum_{S \in \cG_m}2^{m-|S|}\beta^{|S|}A_{\delta+\e_S}
=
\sum_{\{S\to T\} \in \cG_m}2^{m-|S|}\beta^{|S|} (A_{\delta+\e_S} + \beta A_{\delta+\e_T})
\]
where the first sum is over the (sometimes repeated) vertices of the graph $\cG_m$
and second sum is over the edges in $\cG_m$.
In view of Lemma~\ref{cancel} and property (a) in  Lemma~\ref{graph}, 
every term in the last sum is zero so \eqref{key-result-eq} holds.

For the general identity, observe that
if $\lambda$ is a strict partition with $r$ parts then 
$
x_1x_2\cdots x_r A_\lambda = A_{\lambda + 1^r}
$
and
$
x_1x_2\cdots x_r B_\lambda = B_{\lambda + 1^r}
$
where $1^r = (1,1,\dots,1) \in \NN^r$. 
Therefore, setting  $m=q-p+1$ in \eqref{key-result-eq} and multiplying both sides by $(x_1x_2\dots x_{q-p+1})^{p-1}$ gives
$
B_\mu = 2^{q-p+1} A_\mu - 2^{q-p+1}\sum_{i=1}^{q-p+1}  (-\beta/2)^i A_{\mu + \e_{[i]}}
$
which can be rewritten as 
 \eqref{kl2-eq}.
\end{proof}
 
If $\lambda= (\lambda_1 \geq \dots\geq \lambda_p>0)$ and $\mu=(\mu_1\geq \dots \geq \mu_q>0)$ are partitions then 
let $\lambda\mu$ denote their concatenation; this will be another partition if $\lambda_p\geq \mu_1$.
Given a strict partition $\mu$, define $\Lambda(\mu)$ to be the set of strict partitions $\lambda\supseteq \mu$ with
 $\ell(\lambda) = \ell(\mu)$  such that $\SD_{\lambda/\mu}$
is 
a vertical strip. 
We can now prove Theorem~\ref{to-prove}, which states that 
$\GQ_\mu =  2^{\ell(\mu)}  \sum_{\lambda\in \Lambda(\mu)}  (-1)^{\col(\lambda/\mu)}  (-\beta/2)^{|\lambda/\mu| }  \GP_\lambda$.

\begin{proof}[Proof of Theorem~\ref{to-prove}]
Let $\mu=(\mu_1 > \dots >\mu_r>0)$ be a nonempty strict partition. 
We first prove the desired identity specialized to the variables $x_1,x_2,\dots,x_n$, so
assume our fixed value of $n$ has $n\geq r>0$. 
Let $q = \mu_1$ and suppose $m \in [r]$ is maximal with $\mu_m = q + 1-m$.
We proceed by induction on $r-m$.

In the base case when $m=r$, the result to prove is Lemma~\ref{key-lem2}. It remains to
deal with the inductive step. Assume $1 \leq m < r$
and set \[\gamma := (\mu_1 > \mu_2 > \dots > \mu_m)
\quand
\nu:= (\mu_{m+1} > \mu_{m+2} > \dots > \mu_r).\]
Then $\gamma = (q,q-1,q-2,\dots,p)$ for $p := \mu_m$ and $\gamma_m \geq \nu_1+2$.
We may assume by induction that the desired identity holds when $\mu$ is replaced by $\nu$,
since this replacement transforms $r\mapsto r-m$ and $m\mapsto $ (some positive number)
so reduces the difference $r-m$.
 This assumption and Lemma~\ref{key-lem2} imply that $
B_\gamma \GQ_\nu(x_{m+1},\dots,x_n)
$
is equal to
\[
 2^{r}  \sum_{\tilde\gamma  \in \Lambda(\gamma)}\sum_{\tilde\nu  \in \Lambda(\nu)}  (-1)^{\col(\tilde\gamma/\gamma)+ \col(\tilde\nu/\nu)}  (-\beta/2)^{|\tilde\gamma/\gamma| + |\tilde\nu/\nu|}  A_{\tilde\gamma}  \GP_{\tilde\nu}(x_{m+1},\dots,x_n).
\]
Using both parts of Lemma~\ref{key-lem1}, we deduce that 
$
\GQ_\mu(x_1,\dots,x_n)$ is equal to
\[
 2^{r}  \sum_{\tilde\gamma  \in \Lambda(\gamma)}\sum_{\tilde\nu  \in \Lambda(\nu)}   (-1)^{\col(\tilde\gamma/\gamma)+ \col(\tilde\nu/\nu)}  (-\beta/2)^{|\tilde\gamma/\gamma| + |\tilde\nu/\nu|}
 \GP_{\tilde\gamma\tilde\nu}(x_1,\dots,x_n).\]
Since $\mu = \gamma\nu$ and 
$\gamma_m \geq \nu_1+2$,
the  concatenation map $(\tilde\gamma,\tilde\nu) \mapsto \tilde\gamma\tilde\nu$ is a bijection 
 $ \Lambda(\gamma)\times  \Lambda(\nu)\xrightarrow{\sim} \Lambda(\mu)$ and
 if $\lambda = \tilde\gamma\tilde\nu$ for $(\tilde\gamma,\tilde\nu)\in  \Lambda(\gamma)\times  \Lambda(\nu)$ then 
$ \col(\tilde\gamma/\gamma)+ \col(\tilde\nu/\nu) = \col(\lambda/\mu)$
and
$|\tilde\gamma/\gamma| + |\tilde\nu/\nu| = |\lambda/\mu|.$
Hence 
\[
\GQ_\mu(x_1,\dots,x_n) = 2^{r}  \sum_{\lambda \in\Lambda(\mu)}  (-1)^{\col(\lambda/\mu)}  (-\beta/2)^{|\lambda/\mu|}  \GP_{\lambda}(x_1,\dots,x_n).\]
This even holds when $\mu=\emptyset$,
so taking the limit as $n\to\infty$ gives the theorem.
\end{proof}

\section{Weight-preserving bijections}\label{bij-sect}

As the $\GQ_{\lambda}$'s and $\GP_{\lambda}$'s are  generating functions for set-valued shifted tableaux,
Theorem~\ref{to-prove} has some enumerative consequences, which we describe here.

Let $\cX \subseteq \{1,2,3,\dots\}\times \{1,2,3,\dots\}$ be a set of positions. Given a set-valued shifted tableau $T$, define $\unprimemax^\cX(T)$
to be the tableau formed from $T$ by removing the prime from the largest element of $T_{ij}$ for each $(i,j) \in \cX$,
whenever this element is not already primed.
If $\cX =\{(1,1),(1,2),(1,3)\}$ then
\[
\ytableausetup{boxsize=0.8cm,aligntableaux=center}
\unprimemax^\cX\( \begin{ytableau}
  \none & 34' & \none\\
   2'2  & 3' & 34'6'
\end{ytableau}\)
=\begin{ytableau}
  \none & 34' & \none\\
   2'2  & 3 & 34'6
\end{ytableau}
\]
for example. If $T$ is semistandard and $\cX \subseteq \{ (1,1),(2,2),(3,3),\dots \}$
then $\unprimemax^\cX(T)$ is also semistandard. This property may fail if $\cX$ is not a subset of the main diagonal (as we see in the previous example).

If $\lambda$ is a strict partition and  $\cX \subseteq \{ (i,i) : i \in [\ell(\lambda)]\}$ is a set of diagonal positions,
then each $T \in \SetShYTP(\lambda)$ has exactly $2^{|\cX|}$ preimages
in $\SetShYTQ(\lambda)$ under $\unprimemax^{\cX}$.
Given strict partitions $\lambda \supseteq \mu$ define
$\SetShYTP(\lambda : \mu) $ to be the set of semistandard set-valued shifted tableaux $T \in \SetShYTQ(\lambda)$ with 
\[ \unprimemax^{\cX}(T) \in \SetShYTP(\lambda)\text{ for }\cX := \{ (i,i) : \lambda_i = \mu_i\}.\]
The diagonal entry in row $i$ of such a tableau can have 
at most one primed element if $\lambda_i = \mu_i$ and no primed elements if $\lambda_i >\mu_i$.

Finally, for a strict partition $\mu$ let $\Lambda^\pm(\mu)$ be the set of strict partitions $\lambda\supseteq \mu$ with $\ell(\lambda)=\ell(\mu)$ such that $\SD_{\lambda/\mu}$ is a vertical strip and $(-1)^{\col(\lambda/\mu)+|\lambda/\mu|} = \pm 1$.  
Then $\Lambda(\mu) = \Lambda^+(\mu)\sqcup \Lambda^-(\mu)$ and this decomposition reflects the 
decomposition of the right side of Theorem~\ref{to-prove} into positive and negative terms.

\begin{corollary}\label{bijection}
For each strict partition $\mu$ there is a weight-preserving bijection
\[
 \SetShYTQ(\mu) 
 \sqcup \bigsqcup_{\lambda \in \Lambda^-(\mu)}\SetShYTP(\lambda : \mu)
 \to\bigsqcup_{\lambda \in  \Lambda^+(\mu)}\SetShYTP(\lambda:\mu).
 \]
\end{corollary}

\begin{proof}
To see that the domain of the given map is indeed a disjoint union, observe that 
the set  $\Lambda^-(\mu)$ is empty if and only if all parts of $\mu$ differ by at least two.
Since $\col(\mu/\mu)+|\mu/\mu| = 0$, the set $\Lambda^+(\mu)$ is never empty and 
$ \SetShYTQ(\mu) $ and $\SetShYTP(\lambda : \mu)$ are disjoint for all 
$\lambda \in \Lambda^-(\mu)$.

If $\lambda \in \Lambda^\pm(\mu)$ then $\ell(\mu) - |\lambda/\mu| = |\{i : \lambda_i=\mu_i\}|$.
Thus, moving the negative terms on the right side of \eqref{q-to-p-eq} to the left side,
then substituting the generating functions for $\GQ_\mu$ and $\GP_\lambda$ in \eqref{GP-GQ-def}, and finally
 setting $\beta=1$ 
gives
\be\label{bij1}
\sum_{T \in  \SetShYTQ(\mu)} x^T + \sum_{\lambda\in\Lambda^-(\mu)} \sum_{T \in \SetShYTP(\lambda)} 2^{|\{ i : \lambda_i=\mu_i\}|}  x^T \ee
is equal to 
\be\label{bij2}
\sum_{\lambda\in\Lambda^+(\mu)}   \sum_{T \in \SetShYTP(\lambda)}  2^{|\{ i : \lambda_i=\mu_i\}|}   x^T.
\ee 
We have
$
 \sum_{T \in \SetShYTP(\lambda)}  2^{|\{ i : \lambda_i=\mu_i\}|}   x^T =  \sum_{T \in \SetShYTP(\lambda:\mu)}   x^T$
for  any strict partition   $\lambda$,
and the corollary follows by substituting this identity into \eqref{bij1} and \eqref{bij2}
and equating coefficients.
\end{proof}

It is an interesting open problem to find a bijective proof of Theorem~\ref{to-prove}. One way to achieve this would be 
to construct an explicit map realizing Corollary~\ref{bijection}.
This is easy to do when $\mu = (n)$ has only one nonzero part,
in which case the bijection in Corollary~\ref{bijection} is a map
 \be\label{one-row-map}
  \SetShYTQ(n) \to  \SetShYTP(n : n) \sqcup \SetShYTP(n+1).
  \ee
  The set $ \SetShYTP(n : n) $ is contained in   $\SetShYTQ(n)$ and is the union of  $\SetShYTP(n ) $ 
  and the set of tableaux
  formed from elements of $ \SetShYTP(n ) $ by adding a prime to the largest number in box $(1,1)$.
  
  A bijection \eqref{one-row-map}
  is given by mapping each $T \in  \SetShYTP(n : n)  $ to itself
  and each $T \in  \SetShYTQ(n) \setminus  \SetShYTP(n : n)$
  to the tableau in $ \SetShYTP(n+1)$ formed by 
adding $\frac{1}{2}$ to the smallest primed number $i' = i -\frac{1}{2}$ in box $(1,1)$,
  and then splitting this diagonal box into two adjacent boxes containing all numbers $\leq i$
  and $>i$, respectively. This map  is weight-preserving and would send
  \[
\ytableausetup{boxsize=1cm,aligntableaux=center}
\begin{ytableau}
12'3'3 & 34 & 5' 
\end{ytableau}
\quad\mapsto\quad
\begin{ytableau}
12 & 3'3 & 34 & 5'
\end{ytableau}
\]
for example.
   It seems difficult to generalize this idea to larger shapes. 
   Even in the next simplest case $\mu=(3,1)$
we do not know of a straightforward way to describe a weight-preserving bijection
   of the form in Corollary~\ref{bijection}.
   
   In \cite[\S9.2]{IkedaNaruse}, Ikeda and Naruse derive another set of combinatorial 
   formulas for $\GP_\lambda$ and $\GQ_\lambda$ as generating functions for \emph{excited Young diagrams}.
   One could also try to find a bijective proof of Theorem~\ref{to-prove}
   via these expressions.

\section{Skew analogues}
\label{skew-sect}
Let $\mu \subseteq \lambda$ be strict partitions.
A \emph{semistandard (skew) shifted tableau} of shape $\lambda/\mu$ 
is a filling of $\SD_{\lambda/\mu} := \SD_\lambda\setminus\SD_\mu$ by positive half-integers such that rows and columns are weakly increasing,
with no primed entries repeated in a row and no unprimed entries repeated in a column.

Let $\ShYTQ(\lambda/\mu)$ be the set of all such tableaux and let $\ShYTP(\lambda/\mu)$ be the subset 
in which primed entries are disallowed from diagonal positions.
We define both sets to be empty if $\mu \not\subseteq\lambda$.
The \emph{skew Schur $P$- and $Q$-functions} are  
\be 
P_{\lambda/\mu} := \sum_{T\in\ShYTP(\lambda/\mu)} x^T\quand
Q_{\lambda/\mu} := \sum_{T\in\ShYTQ(\lambda/\mu)} x^T
\ee
where as usual $x^T := \prod_{i\geq 1} x_i^{m_i}$ with $m_i$ denoting the number of entries of $T$ equal to $i$ or $i'$.
To motivate these  symmetric functions, we need some additional notation.
If $f \in \ZZ[\beta][[x_1,x_2,\dots]]$,
then  write $f(x,y)$ for the power series $f(x_1,y_1,x_2,y_2,\dots)$
where $x_1,x_2,\dots$  and $y_1,y_2,\dots$ are separate sets of commuting variables;
we also set $f(x) := f(x_1,x_2,\dots)=f$ and $f(y) := f(y_1,y_2,\dots)$.
If $f$ is symmetric then specializing $f(x,y)$ to finitely many variables gives
\[ 
f(x_1,y_1,x_2,y_2,\dots,x_n,y_n) = f(x_1,x_2,\dots,x_n,y_1,y_2,\dots,y_n).
\]
It follows that we can write
\be\label{skew-eq1}
P_\lambda(x,y) = \sum_{\nu  } P_\nu(x) P_{\lambda/\nu}(y)
\quand 
Q_\lambda(x,y) = \sum_{\nu  } Q_\nu(x) Q_{\lambda/\nu}(y)
\ee
where in both sums $\nu$ ranges over all strict partitions \cite[Eq. (8.2)]{Stembridge1989}.

Define the set $\SetShYTQ(\lambda/\mu)$ of \emph{semistandard set-valued (skew) shifted tableaux} of shape $\lambda/\mu$
in the same way as $\SetShYTQ(\lambda)$, just replacing references to ``fillings of $\SD_\lambda$'' by ``fillings of $\SD_{\lambda/\mu}$.''
Let $\SetShYTP(\lambda/\mu)$ be the subset of tableaux in $\SetShYTQ(\lambda/\mu)$ with no primed numbers in any diagonal boxes. 
The \emph{skew $K$-theoretic Schur $P$- and $Q$-functions} are then
\be\label{skew-GP-GQ-def}
\ba
\GP_{\lambda/\mu} &:= \sum_{T\in\SetShYTP(\lambda/\mu)} \beta^{|T| - |\lambda/\mu|} x^T, \\
\GQ_{\lambda/\mu} &:= \sum_{T\in\SetShYTQ(\lambda/\mu)} \beta^{|T| - |\lambda/\mu|}  x^T,
\ea
\ee
where $x^T$ is defined in the same way as for elements of $\SetShYTQ(\lambda)$.
When $\mu \not\subseteq\lambda$ we
 consider both $\SetShYTP(\lambda/\mu)$ and $\SetShYTQ(\lambda/\mu)$ to be empty so that $\GP_{\lambda/\mu}=\GQ_{\lambda/\mu}=0$.
These generalizations of $\GP_\lambda$ and $\GQ_\lambda$
were first defined  in \cite[\S4.6]{LM2021} in the context of \emph{enriched set-valued $P$-partitions}.

The $K$-theoretic version of \eqref{skew-eq1} involves a variant of these power series.
The \emph{removable boxes} of $\mu$
are the positions $(i,j) \in \SD_\mu$ such that $\SD_\mu \setminus\{(i,j)\}$ is the shifted diagram of 
another strict partition.
Let $\Rem(\mu)$ be the set of removable boxes of the strict partition $\mu$.
For strict partitions $\mu\subseteq \lambda$
define 
\be
\ba
\GP_{\lambda\ss \mu} &:= \sum_{\substack{\nu\subseteq \mu,\hs \SD_{\mu/\nu} \subseteq \Rem(\mu)} } \beta^{|\mu/\nu|} \GP_{\lambda/\nu}, \\
\GQ_{\lambda\ss\mu} &:= \sum_{\substack{\nu\subseteq \mu,\hs \SD_{\mu/\nu} \subseteq \Rem(\mu)} } \beta^{|\mu/\nu|} \GQ_{\lambda/\nu},
\ea
\ee
where in both sums $\nu$ must be a strict partition. 
For strict partitions $\mu \not\subseteq\lambda$ we set $\GP_{\lambda\ss\mu}=\GQ_{\lambda\ss\mu}=0$.
Then the definitions of 
$\GP_\lambda$ and $\GQ_\lambda$ as set-valued shifted tableau generating functions imply that
\be\label{skew-eq2}
\ba
\GP_\lambda(x,y) &= \sum_{\nu } \GP_\nu(x) \GP_{\lambda\ss\nu}(y),
\\
\GQ_\lambda(x,y) &= \sum_{\nu } \GQ_\nu(x) \GQ_{\lambda\ss\nu}(y),
\ea
\ee
where both sums are over all strict partitions $\nu$.

Since the $\GP_\nu$'s and $\GQ_\nu$'s are linearly independent and symmetric,
these identities imply that $\GP_{\lambda\ss \mu}$ and $\GQ_{\lambda\ss \mu}$ are symmetric.
Hence $\GP_{\lambda/ \mu}$ and $\GQ_{\lambda/ \mu}$ are also symmetric 
as they can be written in terms of $\GP_{\lambda\ss \mu}$ and $\GQ_{\lambda\ss \mu}$
via inclusion-exclusion \cite[Cor. 5.7]{LM2021}.

More strongly, 
$ \GQ_{\lambda\ss \mu}$ (respectively, $ \GP_{\lambda\ss \mu}$)
is a possibly infinite $\ZZ[\beta]$-linear combination of $\GQ$-functions (respectively, $\GP$-functions) by \cite[Cor 5.13]{LM2021}. Since both functions, when nonzero, are homogeneous of degree
$|\lambda| - |\mu|$ if we set $\deg(\beta) =-1$ and $ \deg(x_i) = 1$, it follows that 
\be
\label{skew-exp-eq}
\ba \GQ_{\lambda\ss \mu} &= \sum_\nu \widehat a_{\mu\nu}^\lambda \beta^{|\mu|+|\nu|-|\lambda|} \GQ_\nu,
\\
\GP_{\lambda\ss \mu} &= \sum_\nu \widehat b_{\mu\nu}^\lambda  \beta^{|\mu|+|\nu|-|\lambda|}\GP_\nu,
\ea \ee
for unique integers $ \widehat a_{\mu\nu}^\lambda,  \widehat b_{\mu\nu}^\lambda \in \ZZ$
which must be zero when $|\mu|+|\nu| < |\lambda|$.
It is expected that these coefficients are all nonnegative, and nonzero  
for only finitely many strict partitions $\nu$;
see \cite[Conj. 5.14]{LM2021}.

Given strict partitions $\mu \subseteq \lambda$,
we consider the statistic
\[\overlap(\lambda/\mu) := |\{ (i,j) \in \SD_{\lambda/\mu} : (i-1,j) \in \SD_{\lambda/\mu}\}|.\]
This quantity is closely related to
 $\col(\lambda/\mu) := | \{ j : (i,j) \in \SD_{\lambda/\mu}\}|$.

\begin{lemma}\label{overlap_cancel}
Let $\mu \subseteq \lambda$ be strict partitions with $\ell(\mu) = \ell(\lambda)$. Then
\[
\sum_{\eta} (-1)^{\col(\lambda/\eta)} 2^{\overlap(\eta/\mu)} =
\sum_{\gamma} (-1)^{\col(\gamma/\mu)} 2^{\overlap(\lambda/\gamma)} = 
\begin{cases}1 &\text{if $\mu=\lambda$} \\ 0&\text{if $\mu \neq \lambda$}\end{cases}
\]
where the first summation is over all strict partitions $\eta$ with $\mu \subseteq \eta \subseteq \lambda$  such that $\SD_{\lambda/\eta}$ is a vertical strip, and the second summation is over all strict partitions $\gamma$ with $\mu \subseteq \gamma \subseteq \lambda$ such that $\SD_{\gamma/\mu}$ is a vertical strip.
\end{lemma}

\begin{proof}
The desired identity is clear if $\mu =\lambda$. Assume $\mu\neq \lambda$ so that $\SD_{\lambda/\mu}$ is nonempty.
We start by showing that the first sum is zero.
Suppose the rightmost box of $\SD_{\lambda/\mu}$ is in column $n$ and $\SD_{\lambda/\mu}$ contains $k>0$ boxes in this column.
Choose a strict partition $\eta$ with $\mu \subseteq \eta \subseteq \lambda$ such that $\SD_{\lambda/\eta}$ is a vertical strip.
Let $L := \{ (i,j) \in \SD_{\eta/\mu} : j<n\}$ and $R  := \{ (i,j) \in \SD_{\eta/\mu} : j=n\}$
so that $\SD_{\eta/\mu} = L  \sqcup R $.
Because $\SD_{\lambda/\eta}$ is a vertical strip, there are only $k+1$ possibilities for
$R$, which must be a set of adjacent positions at the bottom of column $n$ in $\SD_{\lambda/\mu}$. Moreover, 
when $L$ is fixed and $\eta$ varies, each of these possibilities for $R$ occurs exactly once.
Now observe that if $r := |R|$
then 
\[
\ba \overlap(\eta/\mu) &= | \{ (i,j) \in L : (i-1,j)\in L\} | +| \{ (i,j) \in R : (i-1,j)\in R\} |
\\& = | \{ (i,j) \in L : (i-1,j)\in L\} | +\max \{0,r-1\}
\ea
\]
and also
\[
\col(\lambda/\eta) = |\{ j : (i,j) \in \SD_{\lambda/\eta} \sqcup R \}| -\begin{cases} 1 &\text{if }r=k \\ 0&\text{if }0\leq r < k.
\end{cases}
\]
Since $\SD_{\lambda/\eta} \sqcup R= \SD_{\lambda/\mu} - L $, we can rewrite the preceding identity as
\[
\col(\lambda/\eta) = |\{ j : (i,j) \in \SD_{\lambda/\mu} - L \}| + \min \{ 0, k-r-1\}.\]
By substituting these formulas and factoring out the terms depending on $L$,
we deduce that the sum $\sum_{\eta } (-1)^{\col(\lambda/\eta)} 2^{\overlap(\eta/\mu)} $ is a multiple of 
\be\label{zero-sum-eq} \sum_{r=0}^k (-1)^{\min\{0, k-r-1\}} 2^{\max\{0,r-1\}} = 2^{0} + 2^{0} + 2^{1} + \dots + 2^{k-2} - 2^{k-1} = 0\ee
and so is zero itself.

A similar argument shows that the other sum in the lemma is zero. Suppose 
now that the \emph{leftmost} box of $\SD_{\lambda/\mu}$ is in column $n$
and $\SD_{\lambda/\mu}$ contains $k>0$ boxes in this column.
Choose a strict partition $\gamma$ with $\mu \subseteq \gamma \subseteq \lambda$   such that $\SD_{\gamma/\mu}$ is a vertical strip.
Let $L := \{ (i,j) \in \SD_{\lambda/\gamma} : j=n\}$ and $R := \{ (i,j) \in \SD_{\lambda/\gamma} : j>n\}$
so that $\SD_{\lambda/\gamma} = L  \sqcup R $.
Because $\SD_{\gamma/\mu}$ is a vertical strip, there are now only $k+1$ possibilities for $L$, 
which must be a set of adjacent positions at the top of column $n$ in $\SD_{\lambda/\mu}$,
and when $R$ is fixed and $\gamma$ varies, each of these possibilities occurs exactly once. 
Finally, 
if $\ell := |L|$ then we have
\[
\ba
 \overlap(\lambda /\gamma) &= | \{ (i,j) \in R : (i-1,j)\in R\} | +| \{ (i,j) \in L : (i-1,j)\in L\} |
\\&
= | \{ (i,j) \in R : (i-1,j)\in R\} | +\max \{0,\ell-1\}
 \ea
 \]
 and also 
\[
\col(\gamma/\mu) = |\{ j : (i,j) \in \SD_{\gamma/\mu} \sqcup L\}| 
 -\begin{cases} 1 &\text{if }\ell=k \\ 0&\text{if }0\leq \ell < k.\end{cases}
\]
Since $\SD_{\gamma/\mu} \sqcup L=  \SD_{\lambda/\mu} - R $, we can rewrite the preceding identity as
\[
\col(\gamma/\mu) = |\{ j : (i,j) \in \SD_{\lambda/\mu} - R\}|  + \min \{ 0, k-\ell-1\}.\]
By substituting these formulas and factoring out the terms depending on $R$,
we deduce that 
  $\sum_{\gamma} (-1)^{\col(\gamma/\mu)} 2^{\overlap(\lambda/\gamma)}$ is also a multiple of \eqref{zero-sum-eq},
as needed.
  \end{proof}

\begin{remark}
For strict partitions $\lambda$ and $\mu$,
define $M_{\lambda\mu}:=0$ when $\mu\not\subseteq \lambda$ or $\ell(\mu)\neq \ell(\lambda)$, and otherwise set $M_{\lambda\mu} :=  2^{\overlap(\lambda/\mu)}$.
Also define $N_{\lambda\mu} :=0$ when  $\mu\not\subseteq \lambda$ or $\ell(\mu) \neq \ell(\lambda)$ or 
$\SD_{\lambda/\mu}$ is not a vertical strip,  and otherwise set $N_{\lambda\mu} :=  (-1)^{\col(\lambda/\mu)}$.
Lemma~\ref{overlap_cancel} asserts that the matrices
$\left[M_{\lambda\mu}\right]$ and $\left[N_{\lambda\mu}\right]$ are inverses.
\end{remark}

We can now derive a skew generalization of Theorem~\ref{to-prove}.

\begin{theorem}\label{skew-to-prove}
Suppose $\nu\subseteq \mu$ are strict partitions. 
Then
\[   \GQ_{\mu\ss \nu}
=   \sum_{(\kappa,\lambda)} 2^{\ell(\mu)-\ell(\nu)+\overlap(\nu/\kappa)}  (-1)^{\col(\lambda/\mu)}     (-\beta/2)^{|\nu/\kappa|  + |\lambda/\mu|}  \GP_{\lambda \ss\kappa}
\]
where
the sum is over all pairs of strict partitions $(\kappa,\lambda)$ with $\kappa\subseteq \nu \subseteq \mu \subseteq \lambda$ and
 $\ell(\kappa) = \ell(\nu) \leq \ell(\mu) = \ell(\lambda)$
such that $\SD_{\lambda/\mu}$ is a vertical strip.
\end{theorem}

\begin{proof}
Expanding $\GQ_\mu(x,y)$ using Theorem~\ref{to-prove} and then applying \eqref{skew-eq2} gives
\[
\GQ_\mu(x,y) =   \sum_\lambda \sum_{\nu} 2^{\ell(\mu)}  (-1)^{\col(\lambda/\mu)}  (-\beta/2)^{|\lambda/\mu| }  \GP_{\nu}(x) \GP_{\lambda\ss \nu}(y)
\]
where the first sum is over strict partitions $\lambda\supseteq \mu$ with $\ell(\lambda) = \ell(\mu)$ such that $\SD_{\lambda/\mu}$ is a vertical strip and the second sum is over all strict partitions $\nu \subseteq \lambda$.
Alternatively, using Theorem~\ref{to-prove} to expand the right side of \eqref{skew-eq2} gives
\[
\GQ_\mu(x,y)  = \sum_{\eta }   \sum_{\nu}2^{\ell(\nu)}  (-1)^{\col(\nu/\eta)}  (-\beta/2)^{|\nu/\eta| }  \GP_{\nu}(x) \GQ_{\mu \ss \eta}(y)
\] 
where the first sum is over all strict partitions $\eta \subseteq \mu$ and the second sum over strict partitions $\nu\supseteq \eta$ with $\ell(\nu) = \ell(\eta)$ such that $\SD_{\nu/\eta}$ is a vertical strip.\footnote{
When using Theorem~\ref{to-prove} to expand $\GQ_\eta(x)$ in $\GQ_\mu(x,y) = \sum_\eta \GQ_\eta(x) \GQ_{\mu\ss\eta}(y)$
one expects to see the factor $2^{\ell(\eta)}$, but this can be changed to $2^{\ell(\nu)}$ since $\ell(\nu) = \ell(\eta)$.}
Equating the coefficients of $\GP_\nu$ in these expressions gives
\[
 \sum_{\eta} 2^{\ell(\nu)} (-1)^{\col(\nu/\eta)} (\tfrac{-\beta}{2})^{|\nu/\eta|} \GQ_{\mu\ss \eta}
=
 \sum_{\lambda} 2^{\ell(\mu)} (-1)^{\col(\lambda/\mu)} (\tfrac{-\beta}{2})^{|\lambda/\mu|} \GP_{\lambda \ss\nu}
\]
where the sums are over certain strict partitions $\eta$ and $\lambda$;
to be precise,
assuming $\nu \subseteq \mu$,
the preceding equation is equivalent to
\[
\GQ_{\mu\ss\nu} = 
 \sum_{\lambda} 2^{\ell(\mu)-\ell(\nu)} (-1)^{\col(\lambda/\mu)} (\tfrac{-\beta}{2})^{|\lambda/\mu|} \GP_{\lambda \ss\nu}
-
\sum_{\eta\subsetneq\nu}  (-1)^{\col(\nu/\eta)} (\tfrac{-\beta}{2})^{|\nu/\eta|} \GQ_{\mu\ss \eta}
\]
where the first sum is
over
strict partitions $\lambda \supseteq \mu $
with $\ell(\lambda) = \ell(\mu)$ such that $\SD_{\lambda/\mu}$  is a vertical strip
and the second sum is 
 over strict partitions $\eta \subsetneq \nu$ with $\ell(\eta) = \ell(\nu)$
such that $\SD_{\nu/\eta}$ is a vertical strip.

When $\nu = (m,m-1,\dots,2,1)$ for some $m \in \NN$, the sum over $\eta$ has zero terms and the preceding formula
reduces to the desired identity.
Otherwise, we may assume by induction that the desired formula holds for each $\GQ_{\mu\ss \eta}$.
Substituting these formulas into the displayed equation
gives an expression for $\GQ_{\mu\ss\nu}$ as a linear combination of 
$\GP_{\lambda\ss\kappa}$'s where $\kappa$ and $\lambda$ range over all strict partitions with  $\kappa\subseteq \nu \subseteq \mu \subseteq \lambda$  and
$\ell(\kappa) = \ell(\nu) \leq \ell(\mu) = \ell(\lambda)$
such that $\SD_{\lambda/\mu}$ is a vertical strip. 
The coefficient of $\GP_{\lambda\ss\nu}$ in this expansion is  $2^{\ell(\mu)-\ell(\nu)}(-1)^{\col(\lambda/\mu)}(\frac{-\beta}{2})^{|\lambda/\mu|}$ as desired. The coefficient of   $\GP_{\lambda\ss\kappa}$ when $\kappa \subsetneq \nu\subseteq \mu \subseteq \lambda$ 
and $\ell(\kappa) = \ell(\nu) \leq \ell(\mu) = \ell(\lambda)$ and $\SD_{\lambda/\mu}$ is a vertical strip is  
\[-
\sum_{\eta} (-1)^{\col(\nu/\eta)}(\tfrac{-\beta}{2})^{|\nu/\eta|}\(2^{\ell(\mu)-\ell\(\eta\)+\overlap(\eta/\kappa)}(-1)^{\col(\lambda/\mu)}(\tfrac{-\beta}{2})^{|\eta/\kappa|+|\lambda/\mu|}\)
\]
where the sum is
 over all strict partitions $\eta$ with $\kappa \subseteq \eta \subsetneq \nu$ such that $\SD_{\nu/\eta}$ is a vertical strip.
Since then $\ell(\kappa) = \ell(\eta) = \ell(\nu)$,
we can rewrite this as
\[2^{\ell(\mu)-\ell\(\nu\)}  (-1)^{\col(\lambda/\mu)}  (\tfrac{-\beta}{2})^{|\lambda/\mu|+|\nu/\kappa|}  \sum_{\eta } (-1)^{\col(\nu/\eta)+1}  2^{\overlap(\eta/\kappa)},\] so it suffices to show   that
$\sum_{\eta} (-1)^{\col(\nu/\eta)+1}  2^{\overlap(\eta/\kappa)} = 2^{\overlap(\nu/\kappa)}$
where the sum is again over $\eta$ with $\kappa \subseteq \eta \subsetneq \nu$ such that $\SD_{\nu/\eta}$ is a vertical strip.
After moving all terms to one side, this identity is half of Lemma~\ref{overlap_cancel}.
\end{proof}

We mention one corollary, which was noted in passing above. Let $\delta_m := (m,m-1,\dots,2,1)$ for $m \in \NN$. 
The following recovers Theorem~\ref{to-prove} when $m=0$.

\begin{corollary}
Suppose $\mu$ is a strict partition with  $\delta_m \subseteq \mu$. Then
\[ \GQ_{\mu\ss \delta_m}
=
2^{\ell(\mu)-m} \sum_{\lambda}  (-1)^{\col(\lambda/\mu)} (-\beta/2)^{|\lambda/\mu|} \GP_{\lambda \ss\delta_m}
\]
where the sum is over strict partitions $\lambda\supseteq \mu$
with $\ell(\lambda) = \ell(\mu)$ such that $\SD_{\lambda/\mu}$ is a vertical strip.
\end{corollary}


\begin{proof}
If $\nu =\delta_m$ in Theorem~\ref{skew-to-prove}, then the only strict partition $\kappa \subseteq \nu$
with $\ell(\kappa) = \ell(\nu)$ is $\kappa=\nu=\delta_m$, which has $\overlap(\nu/\kappa) = |\nu/\kappa| = 0$.
\end{proof}

There is some interest in determining when there are coincidences 
$P_{\lambda/\mu} = P_{\nu/\kappa}$ and $Q_{\lambda/\mu} = Q_{\nu/\kappa}$
among the skew Schur $P$- and $Q$-functions \cite{BW09,GillespieSalois,Salmasian}. This phenomenon is less well-understood than for  
 skew Schur functions.

One could consider the same problem for the skew $\GP$- and $\GQ$-functions.
In particular, any equality $\GP_{\lambda/\mu} = \GP_{\nu/\kappa}$
would imply that $P_{\lambda/\mu} = P_{\nu/\kappa}$ (and likewise for the $Q$-functions),
but it is not clear if the converse always holds.
With these questions in mind, we explain
one nontrivial equality between Schur $Q$-functions that generalizes to the $K$-theoretic setting.

Define the \emph{flip} of a skew shape $\lambda/\mu$ to be the skew shape $\phi(\lambda/\mu)$ whose shifted diagram $\SD_{\phi(\lambda/\mu)}$ is formed by reflecting $\SD_{\lambda/\mu}$ across a line perpendicular to the main diagonal, so that in French notation the bottom row becomes the rightmost column. We refer to this operation as ``flipping the diagram $\SD_{\lambda/\mu}$'':
\[
\ytableausetup{boxsize=0.5cm,aligntableaux=center}
\barr{c} \begin{ytableau}
\none  & \none  & \\
\none  &  & \\
\none[\cdot] & \none[\cdot] &  & 
\end{ytableau}
\\[-7pt]\\
\lambda/\mu = (4,2,1)/(2)
\earr
\quad\text{flips to}\quad
\barr{c} \begin{ytableau}
\none & \none & \\
\none & & & \\
\none[\cdot] & \none[\cdot] & \none[\cdot] &
\end{ytableau}
\\[-7pt]\\
\phi(\lambda/\mu) = (4,3,1)/(3).
\earr
\]
The following reduces to \cite[Prop. IV.13]{DeWitt} when $\beta=0$:

\begin{proposition}\label{GQ_equal}
Let $\mu \subseteq \lambda$ be strict partitions, then 
\[\GP_{\phi(\lambda/\mu)} = \GP_{\lambda/\mu} \quad\text{and}\quad \GQ_{\phi(\lambda/\mu)} = \GQ_{\lambda/\mu}.
\]
\end{proposition}

\begin{proof}
We first prove the $\GQ$-identity.
Suppose $T \in \SetShYTQ(\lambda/\mu)$.
Write $\min(T)$ and $\max(T)$ for the minimal and the maximal numbers appearing in any entry of $T$. 
Let $n:=\lceil\min(T)\rceil$ be whichever of $\min(T)$ or $\min(T)+\frac{1}{2}$ is an integer,
and define $N:=\lceil\max(T)\rceil$ similarly. Now let $\phi_Q(T)$ be the set-valued shifted tableau of shape $\phi(\lambda/\mu)$
formed by flipping $T$ and then replacing each number $a$ in each set-valued entry by $n+N-\frac{1}{2}-a$. 

The rows and columns of $\phi_Q(T)$ are weakly increasing since the rows and columns of $T$ are weakly increasing. As exactly one of $a$ or $n+N-\frac{1}{2}-a$ is primed, the tableau $\phi_Q(T)$ does not have any unprimed numbers repeated in a column or primed numbers repeated in a row.
We conclude that  $\phi_Q$ defines a map $\SetShYTQ(\lambda/\mu) \to \SetShYTQ(\phi(\lambda/\mu))$, which is clearly invertible.
 
 Notice that the weight of $\phi_Q(T)$ is $x^{\phi_Q(T)} = \sigma(x^T)$, where $\sigma$ is the permutation with $\sigma(a) = n+N-a$ for all $n \leq a \leq N$. Since the values of $n$ and $N$ are determined by the monomial $x^T$,
 the symmetry of
 $\GQ_{\lambda/\mu}$ implies that \be\label{calc-eq}\GQ_{\lambda/\mu} = \sum_{T\in\SetShYTQ(\lambda/\mu)} \beta^{|\phi_Q(T)|-|\lambda/\mu|} x^{\phi_Q(T)}= \GQ_{\phi(\lambda/\mu)}.\ee

The proof of the $\GP$-identity is similar, except now 
for $T \in \SetShYTP(\lambda/\mu)$ we define $\phi_P(T)$ from  $\phi_Q(T)$
 by adding $\frac{1}{2}$ to all numbers in diagonal positions.
 Since all numbers in diagonal entries of $\phi_Q(T)$ are primed when $T \in \SetShYTP(\lambda/\mu)$,
 we have $x^{\phi_P(T)} = x^{\phi_Q(T)}$ and
 it follows that $\phi_P$ is a bijection $\SetShYTP(\lambda/\mu) \to \SetShYTP(\phi(\lambda/\mu))$.
 We can therefore replace every letter ``$Q$'' in \eqref{calc-eq} by ``$P$''
 to deduce that $ \GP_{\lambda/\mu} =\GP_{\phi(\lambda/\mu)}$.
\end{proof}

It is not clear if there is a meaningful way to extend the preceding result to $\GP_{\lambda\ss\mu}$ and $\GQ_{\lambda\ss\mu}$.
If one defines $\phi(\lambda\ss\mu) := \tilde \lambda \ss \tilde \mu$ where $\phi(\lambda/\mu)= \tilde\lambda/\tilde \mu$,
for example,
then it may hold that
$\GP_{\phi(\lambda\ss\mu)} \neq \GP_{\lambda\ss\mu} $ and $\GQ_{\phi(\lambda\ss\mu)} \neq \GQ_{\lambda\ss\mu}$.

\section{Dual functions}\label{dual-sect}

Let $\overline{x} := \frac{-x}{1+\beta x}$ so that $x \oplus \overline x = 0$.
Nakagawa and Naruse \cite{NakagawaNaruse} define
the \emph{dual $K$-theoretic Schur $P$- and $Q$-functions} $\gp_\lambda$ and $\gq_\lambda$
 to be the unique elements of $\ZZ[\beta][[x_1,x_2,\dots]]$ indexed by strict partitions $\lambda$ 
 such that
 \be\label{cauchy-eq} \sum_\lambda \GQ_\lambda(x) \gp_\lambda(y) = \sum_\lambda \GP_\lambda(x) \gq_\lambda(y) = \prod_{i,j \geq 1} \frac{ 1 - \overline{x_i} y_j}{1-x_iy_j}.\ee
Both sums in this \emph{Cauchy identity} are over all strict partitions $\lambda$. 

The power series  $\gp_\lambda$ and $\gq_\lambda$ are a special case of the \emph{dual universal factorial Schur $P$- and $Q$-functions}
 in \cite[Def. 3.2]{NakagawaNaruse}.
One reason these specialization are interesting (compared to the more general ``universal'' functions) is because 
they have conjectural formulas as generating functions for
 \emph{shifted reverse plane partitions}
\cite[Conj 5.1]{NakagawaNaruse}.  
We will discuss this idea in Section~\ref{last-sect}.

Both $\gp_\lambda$ and $\gq_\lambda$ are symmetric in the $x_i$ variables
and of degree $|\lambda|$ if we set $\deg(\beta) =0$ and $\deg(x_i) = 1$ \cite[\S3.2]{NakagawaNaruse}.
The sets  $\{ \gp_\lambda :\lambda\text{ is a strict partition} \}$ and  $\{\gq_\lambda  :\lambda\text{ is a strict partition}\}$ 
are $\ZZ[\beta]$-bases for subrings of  $\ZZ[\beta][[x_1,x_2,\dots]]$ by  \cite[Thm. 3.1]{NakagawaNaruse},
and it holds that
$
P_{\lambda} = \gp_{\lambda}|_{\beta=0} 
$
and
$
Q_{\lambda} =\gq_{\lambda}|_{\beta=0} 
$ \cite[\S3.2]{NakagawaNaruse}.

\begin{proposition}\label{recover-prop}
We recover $\gp_\lambda$ from $\gp_\lambda|_{\beta=1}$ (respectively, $\gq_\lambda$ from $\gq_\lambda|_{\beta=1}$)
by substituting $x_i \mapsto \beta^{-1} x_i$ for all $i$ and then multiplying by $\beta^{|\lambda|}$.
As such,  if   we set $\deg(\beta)=\deg(x_i)=1$ then
$ \gp_{\lambda}$ and $\gq_{\lambda}$ are homogeneous
of degree $|\lambda|$.
\end{proposition}

\begin{proof}
We recover the original form of $\prod_{i,j \geq 1} \tfrac{ 1 - \overline{x_i} y_j}{1-x_iy_j}$ after setting $\beta=1$
by substituting $x_i \mapsto \beta x_i$ and $y_j \mapsto \beta^{-1} y_j$ for all $i$ and $j$.
It follows that if we define $\gp^{(1)}_\lambda:=\gp_\lambda|_{\beta=1}$ and $\GQ^{(1)}_\lambda := \GQ_\lambda|_{\beta=1}$
then 
\[
\ba  \sum_\lambda \GQ^{(1)}_\lambda(\beta x_1, \beta x_2,\dots) \gp^{(1)}_\lambda(\beta^{-1} y_1, \beta^{-1} y_2,\dots) 
&= \prod_{i,j \geq 1} \tfrac{ 1 - \overline{x_i} y_j}{1-x_iy_j} 
\\&= \sum_\lambda \GQ_\lambda(x) \gp_\lambda(y).
\ea\] 
As it is clear from \eqref{GP-GQ-def} that  $\GQ_\lambda(x) = \beta^{-|\lambda|}\GQ^{(1)}_\lambda(\beta x_1, \beta x_2,\dots)$,
this equation can only hold if $ \gp_\lambda(y)= \beta^{|\lambda|} \gp^{(1)}_\lambda(\beta^{-1} y_1, \beta^{-1} y_2,\dots)$ 
which is equivalent to what is claimed in the lemma.
The identity for $\gq_\lambda$ follows similarly.
\end{proof}

Theorem~\ref{to-prove} has a dual version that gives a $\gp$-expansion of  $\gq_\lambda$:

\begin{corollary}\label{to-prove2} If $\lambda$ is a strict partition then
\[
\gq_\lambda =  2^{\ell(\lambda)} \sum_{\mu}   (-1)^{\col(\lambda/\mu)}  (-\beta/2)^{|\lambda/\mu| }  \gp_\mu
\]
where the sum is over strict partitions $\mu \subseteq \lambda$ with $\ell(\mu) = \ell(\lambda)$
such that $\SD_{\lambda/\mu}$ is a vertical strip.
\end{corollary}

\begin{proof}
Expand $\GQ_\mu(x)$ in 
$ \sum_\mu \GQ_\mu(x) \gp_\mu(y) = \sum_\lambda \GP_\lambda(x) \gq_\lambda(y)$
using Theorem~\ref{to-prove}
and then equate the coefficients of $\GP_\lambda(x)$.
\end{proof}

For example,  $\gq_{(m,\dots,3,2,1)} = 2^m\gp_{(m,\dots,3,2,1)}$ for any $m\in \NN$ and $\gq_{(n)} = 2 \gp_{(n)} + \beta \gp_{(n-1)}$ when $n \geq 2$.
There is also an analogue of Corollary~\ref{positive_case}:

\begin{corollary}
If $\lambda$ is a strict partition with $m$ parts then 
 $\gq_\lambda$ is an $\NN[\beta]$-linear combination of $\gp$-functions
if and only if $\lambda -  (m,\dots,3,2,1)$ is also strict.
\end{corollary}

\begin{proof}
One can check that if $\lambda -  (m,\dots,3,2,1)$ is strict
then every $\mu$ indexing the sum in Corollary~\ref{to-prove2} has $\col(\lambda/\mu) = |\lambda/\mu|$,
while if $\lambda -  (m,\dots,3,2,1)$ is not strict then at least one such $\mu$ has $\col(\lambda/\mu) =1$ and $|\lambda/\mu|=2$.
\end{proof}

Recall that $\widehat a_{\mu\nu}^\lambda$ and $\widehat b_{\mu\nu}^\lambda$ are
the integers
appearing in the respective expansions of $\GQ_{\lambda\ss\mu}$ and $\GP_{\lambda\ss\mu}$ in  \eqref{skew-exp-eq}. These numbers are zero if $|\mu|+|\nu| < |\lambda|$.

\begin{proposition}\label{ab-prop1}
If $\mu$ and $\nu$ are strict partitions then  
\be\label{dual-exp-eq2} \gp_\mu \gp_\nu = \sum_{\lambda} \widehat a_{\mu\nu}^\lambda \beta^{|\mu|+|\nu|-|\lambda|} \gp_{\lambda}
\text{ and }
\gq_\mu \gq_\nu = \sum_{\lambda} \widehat b_{\mu\nu}^\lambda  \beta^{|\mu|+|\nu|-|\lambda|} \gq_{\lambda}
\ee
where the sums are over all strict partitions $\lambda$. 
\end{proposition}

\begin{proof}
This follows by a standard argument similar to the proof of \cite[Prop. 8.2]{Stembridge1989}, which is equivalent to 
\eqref{dual-exp-eq2} when $\beta=0$. 
Let $\Delta(x;y) := \prod_{i,j \geq 1} \frac{ 1 - \overline{x_i} y_j}{1-x_iy_j}$.
Introduce a third set of commuting variables $z_1,z_2,z_3,\dots$. Then
\[\ba
 \sum_{\lambda}\sum_{\mu} \GQ_\mu(x)\GQ_{\lambda\ss \mu}(y) \gp_\lambda(z) 
&= \sum_\lambda \GQ_\lambda(x,y) \gp_\lambda(z) = \Delta(x,z) \Delta(y,z) \\&= \(\sum_\mu \GQ_\mu(x) \gp_\mu(z)\)\(\sum_\nu \GQ_\nu(y) \gp_\nu(z)\)
\ea\]
by equations \eqref{skew-eq2} and \eqref{cauchy-eq}.
 The first identity follows by substituting the formula \eqref{skew-exp-eq} for $\GQ_{\lambda\ss \mu}$ into the first expression
and extracting the coefficients of $\GQ_\mu(x) \GQ_\nu(y)$.
 The second identity follows similarly.
\end{proof}

Since 
 $\GP_{\lambda\ss\mu} = \GQ_{\lambda\ss\mu}=0$ when $\mu\not\subseteq\lambda $,
we already know from \eqref{skew-exp-eq} that 
\be \widehat a_{\mu\nu}^\lambda =\widehat b_{\mu\nu}^\lambda =0\ee when 
$\mu \not\subseteq \lambda$ (and, as noted above, when $|\lambda| > |\mu| + |\nu| $).
We also have
\be\widehat a_{\mu\nu}^\lambda  =\widehat a_{\nu\mu}^\lambda  
\quand
\widehat b_{\mu\nu}^\lambda  =\widehat b_{\nu\mu}^\lambda \ee
by Proposition~\ref{ab-prop1}
since $\ZZ[\beta][[x_1,x_2,\dots]]$ is a commutative ring.

If $\mu$ and $\nu$ are strict partitions
then \cite[Props. 3.4 and 3.5]{IkedaNaruse} imply that  
\be
\label{skew-exp-eq2}
\ba
\GP_\mu \GP_\nu &= \sum_\lambda  a_{\mu\nu}^\lambda \beta^{|\lambda|-|\mu|-|\nu|} \GP_\lambda,
\\
\GQ_\mu \GQ_\nu &= \sum_\lambda  b_{\mu\nu}^\lambda  \beta^{|\lambda|-|\mu|-|\nu|} \GQ_\lambda,
\ea\ee
for unique integers $  a_{\mu\nu}^\lambda ,   b_{\mu\nu}^\lambda  \in \ZZ$.
 It is known from \cite{CTY} that the $\GP$-expansion is finite with every $a_{\mu\nu}^\lambda  \in \NN$,
 but \emph{a priori} the coefficients $b_{\mu\nu}^\lambda$ could  be nonzero for infinitely many strict partitions $\lambda$.
We will discuss the  problem of showing that the $\GQ$-expansion is finite in Section~\ref{last-sect}.
It is again clear that  
\be\label{ab-ba-eq}  a_{\mu\nu}^\lambda =   a_{\nu\mu}^\lambda 
\quand b_{\mu\nu}^\lambda  =   b_{\nu\mu}^\lambda \ee
but the following property requires an argument.

\begin{proposition}\label{contain-bound-lem}
One has $a^\lambda_{\mu\nu} = b^\lambda_{\mu\nu} =0$
if $\mu \not\subseteq \lambda$ or $\nu\not \subseteq\lambda$ or $|\mu| + |\nu|>|\lambda|$.
\end{proposition}

\begin{proof}
\cite[Prop. 3.4]{IkedaNaruse} asserts that any finite product of $\GP$-functions
is a (possibly infinite) $\ZZ[\beta]$-linear combination of $\GP$-functions, while
\cite[Prop. 3.5]{IkedaNaruse} states an analogous property for the $\GQ$-functions.
Since $\GP_\lambda$ and $\GQ_\lambda$ are both homogeneous of degree $|\lambda|$
if we set $\deg(\beta)=-1$ and $\deg(x_i)=1$, it must hold that $a^\lambda_{\mu\nu} = b^\lambda_{\mu\nu} =0$
when  $|\mu| + |\nu|>|\lambda|$.

It remains to show that both coefficients are zero when $\mu \not\subseteq \lambda$.
This follows for $a_{\mu\nu}^\lambda$ from the Littlewood--Richardson rule for the $\GP_{\lambda}$'s 
due to Clifford, Thomas,  and Yong \cite[Thm. 1.2]{CTY}:
in this result, the coefficient $C^{\lambda}_{\mu\nu}$ is equal to $a_{\mu\nu}^\lambda$ by \cite[Cor. 8.1]{IkedaNaruse}.

We turn to the $b^{\lambda}_{\mu\nu}$ coefficients. 
When $p$ is a positive integer, there is a Pieri-type formula for the $\GQ$-expansion 
of $\GQ_\mu \GQ_{(p)}$
due to Buch and Ravikumar \cite[Cor. 5.6]{BuchRavikumar}: in this result,
 the coefficient $c^\lambda_{\mu,p}$ is equal to $b^{\lambda}_{\mu\nu}$ for $\nu=(p)$ again
 by \cite[Cor. 8.1]{IkedaNaruse}.
Buch and Ravikumar's formula implies that $b^{\lambda}_{\mu\nu} = 0$ if $\mu \not\subseteq\lambda$ and 
$\ell(\nu) \leq 1$.
It follows that if $p_1,p_2,\dots,p_k$ are any positive integers then 
the product  $\GQ_\mu \GQ_{(p_1)}\GQ_{(p_2)}\cdots \GQ_{(p_k)}$
is a possibly infinite $\ZZ[\beta]$-linear combination of 
$\GQ_\lambda$'s indexed by strict partitions with $\lambda \supseteq \mu$.

From this observation, to complete the proof it is enough to show that each $\GQ_\nu$
is a possibly infinite $\ZZ[\beta]$-linear combination of products of the form
$\GQ_{(p_1)}\GQ_{(p_2)}\cdots \GQ_{(p_k)}$.
This claim is a consequence of  \cite[Thm. 5.8]{NakagawaNaruse0},
which gives a formula 
for a more general \emph{universal factorial Hall--Littlewood $Q$-function} $\HQ^{\mathbb{L}}_\nu(\mathbf{x}_n;t\hs | \hs \mathbf{b})$,
essentially as a linear combination of products of other such functions indexed by one-row partitions.
The notation in \cite{NakagawaNaruse0} makes it slightly nontrivial to connect this formula
to our situation.
However, the property needed for $\GQ_\nu$ follows directly from the discussion in \cite[\S5.2.4]{NakagawaNaruse0}
(in particular, from \cite[Thm. 5.10]{NakagawaNaruse0}),
 once one observes that the generating function 
 \[ \GQ(u\hs |\hs \mathbf{b})^{(k)} := \frac{1}{1+\beta u} \prod_{j=1}^n \frac{1 + (u^{-1}+ \beta)x_j}{1+ (u^{-1} + \beta)\overline{x_j}} \times \prod_{j=1}^k (1 + (u^{-1}+\beta)b_j)\]
  in
  \cite[Eq.\ (5.10)]{NakagawaNaruse0} reduces when $\mathbf{b} = \mathbf{0}$ and $n\to \infty$
  to the expression below:

\begin{lemma} When expanded as a power series in  $u^{-1}$, the expression
\[ \frac{1}{1+\beta u} \prod_{j=1}^\infty \frac{ 1 + (u^{-1}+\beta) x_j}{1 + (u^{-1} + \beta) \overline{x_j}}\]
is equal to $\sum_{n \in \ZZ} \GQ_{(n)} u^{-n}$ where we set $\GQ_{(n)} := (-\beta)^{-n}$ for $n<0$.
\end{lemma}

\begin{proof}
This is essentially \cite[Rem. 5.11]{Hud} given \cite[Def. 3.5]{Hud}.
Here is a self-contained proof explained to us by Joel Lewis.
Let $\ell(T)$ denote the number boxes in a   tableau $T$.
Then  $\sum_{n> 0} \GQ_{(n)}t^n = \sum_{T} \beta^{|T|-\ell(T)} x^{T} t^{\ell(T)}$ where the sum is over 
all semistandard set-valued shifted tableaux $T$ with a nonempty one-row shape.
Such a tableau $T$ is specified uniquely by the following choices:
\begin{itemize}
\item The finite set of positive integers $\{j_1< \ldots< j_k\}$ such that at least one of $j_i$ or $j_i'$ appears in some box of $T$.
\item The numbers $n_1, \ldots, n_k>0$ such that $j_i$ or $j_i'$ appears in $n_i$ boxes of $T$.
\item For each $i \in [k]$, whether the first box containing $j_i$ or $j_i'$ contains just $j_i$, just $j_i'$, or both $j_i$ and $j_i'$.
\item For each $i>1$, whether the first box containing $j_i$ or $j_i'$ contains no smaller numbers, or contains at least one of $j_{i-1}$ or $j_{i-1}'$.
\end{itemize}
For the tableau $T$ corresponding to this data, the value of $|T|$ (respectively, $\ell(T)$) depends on the numbers $n_1,\ldots,n_k$
and the choices in the third (respectively, fourth) bullet point.
It follows that we can write   $ \sum_{n>0} \GQ_{(n)} t^n$ as  
\[
 \sum_{\substack{ k>0 \\ 0<j_1 < \dots < j_k \\ n_1, \dots, n_k > 0 }} (2  x_{j_1} + \beta x_{j_1}^2)t \times (x_{j_1} t)^{n_1 - 1}  
\times \prod_{i = 2}^k (2\beta x_{j_i} + \beta^2 x_{j_i}^2)(\tfrac{t}{\beta}  + 1) \times (x_{j_i} t)^{n_i - 1}
\]
Fixing $k$ and $j_1<\dots<j_k$ and 
summing over $n_1,\dots,n_k$ turns this into
\[
\sum_{\substack{ k>0 \\ 0<j_1 < \dots < j_k}} \frac{(2x_{j_1} + \beta x_{j_1}^2)t}{ 1 - x_{j_1} t}  
\times \prod_{i = 2}^k \frac{(2x_{j_i} + \beta x_{j_i}^2)(t + \beta)}{1 - x_{j_i} t}.
\]
Pulling out a factor of $\frac{t}{t+\beta} = \frac{1}{1+\beta t^{-1}}$ transforms this to
\[
  \sum_{\substack{S \subseteq \{1,2,\dots\} \\ 0<|S| <\infty}} \tfrac{t}{t+\beta}\prod_{j \in S} \tfrac{(2x_{j} + \beta x_{j}^2)(t + \beta)}{1 - x_{j} t}
=
 \tfrac{1}{1+\beta t^{-1}}\(\prod_{j=1}^\infty \(1 + \tfrac{(2x_{j} + \beta x_{j}^2)(t + \beta)}{1 - x_{j} t} \) - 1\).
\]
Since
$
1 + \frac{(2x + \beta x^2)(t + \beta )}{1 - x t} = \frac{(1 + \beta x)(1 + \beta x + xt)}{1 - xt} = \frac{1 + (t+\beta)x}{1 + (t+\beta)\overline{x}}$
we conclude that 
\[\sum_{n>0} \GQ_{(n)} t^n = \frac{1}{1+\beta t^{-1}}\prod_{j=1}^\infty \frac{1 + (t+\beta)x_j}{1 + (t+\beta)\overline{x_j}} - \frac{1}{1+\beta t^{-1}}.\]
The desired formula follows by replacing the formal parameter $t$ by $u^{-1}$.
\end{proof}
\noindent
This completes the proof of Proposition~\ref{contain-bound-lem}.
\end{proof}

For any strict partitions $\lambda$ and $\mu$ we define
\be\label{skew-g-def}
\gq_{\lambda/\mu} := \sum_{\nu}  a_{\mu\nu}^\lambda  \beta^{|\lambda|-|\mu|-|\nu|}  \gq_\nu
\text{ and }
\gp_{\lambda/\mu} := \sum_{\nu}  b_{\mu\nu}^\lambda  \beta^{|\lambda|-|\mu|-|\nu|}  \gp_\nu
\ee
where $\nu$ ranges over all strict partitions.
Proposition~\ref{contain-bound-lem} implies that these symmetric functions have bounded degree
and are zero whenever $\mu\not\subseteq\lambda$.
As $\gp_\lambda$ and $\gq_\lambda$ are homogeneous of degree $|\lambda|$ when $\deg(\beta)= 1$,
it follows that $\gp_{\lambda/\mu}$ and $\gq_{\lambda/\mu}$
are homogeneous of degree $|\lambda|-|\mu|$ when nonzero.

\begin{proposition}
If $\lambda$ is a strict partition then
\be\label{skew-eq3}
\gp_\lambda(x,y) = \sum_\nu \gp_\nu(x) \gp_{\lambda/\nu}(y)
\quand
\gq_\lambda(x,y) = \sum_\nu \gq_\nu(x) \gq_{\lambda/\nu}(y)
\ee
where both sums are over all strict partitions $\nu$.
\end{proposition}

\begin{proof}
These identities are equivalent to the first two parts of \cite[Prop. 3.2]{NakagawaNaruse}.
We show how to derive the first identity for completeness. 
Observe that 
\[\ba
  \sum_\lambda \GQ_\lambda(x) \gp_\lambda(y,z) &= \Delta(x,y) \Delta(x,z) 
  \\&= \(\sum_\mu \GQ_\mu(x) \gp_\mu(y)\)\(\sum_\nu \GQ_\nu(x) \gp_\nu(z)\)
  \\& =   \sum_{\mu, \nu,\lambda} b_{\mu\nu}^\lambda \beta^{|\lambda|-|\mu|-|\nu|}  \GQ_\lambda(x)  \gp_\mu(y)  \gq_\nu(z)
    \\& =\sum_{\lambda}  \(\GQ_\lambda(x)\sum_\mu \gp_\mu(y) \gp_{\lambda/\mu}(z)\)
\ea\]
by \eqref{cauchy-eq} and \eqref{skew-g-def}.
Now equate the coefficients of $\GQ_\lambda(x)$. 
\end{proof}
 
 \begin{remark}
 Setting $\beta=0$ transforms \eqref{skew-eq3} to \eqref{skew-eq1}, 
so  
$P_{\lambda/\mu} = \gp_{\lambda/\mu}|_{\beta=0}$
and
$Q_{\lambda/\mu} =\gq_{\lambda/\mu}|_{\beta=0}.$
Thus $\gp_{\lambda/\mu}$ and $\gq_{\lambda/\mu}$ are each nonzero if and only if $\mu \subseteq\lambda$.
\end{remark}

We may now prove a dual version of Theorem~\ref{skew-to-prove}. 

\begin{theorem}\label{skew-to-prove2}
Suppose $\kappa\subseteq \lambda$  are strict partitions. Then
\[   \gq_{\lambda/\kappa} = \sum_{(\nu,\mu)} 2^{\ell(\lambda)-\ell(\kappa)+\overlap(\nu/\kappa)}  (-1)^{\col(\lambda/\mu)}     (-\beta/2)^{|\nu/\kappa|  + |\lambda/\mu|}   \gp_{\mu/\nu}
\]
where the sum is over all pairs of strict partitions $(\nu,\mu)$ with $\kappa\subseteq \nu \subseteq \mu \subseteq \lambda$ and
$\ell(\kappa) = \ell(\nu) \leq \ell(\mu) = \ell(\lambda)$
such that $\SD_{\lambda/\mu}$ is a vertical strip.
\end{theorem}

\begin{remark}
If there were a bilinear form that made $\{ \gq_{\lambda/\kappa} \}$ the dual basis of $\{ \GP_{\lambda/\kappa} \}$
and $\{ \gp_{\mu/\nu} \}$ the dual basis of $\{ \GQ_{\mu/\nu} \}$,
then this result would follow immediately from Theorem~\ref{skew-to-prove}.
Since the various skew functions are not even linearly independent, no such form is available
 and we need another argument.
\end{remark} 

\begin{proof}

Expanding $\gq_\lambda(x,y)$ using Corollary~\ref{to-prove2} and then applying \eqref{skew-eq3} gives
\[
\gq_\lambda(x,y) = \sum_\mu \sum_{\kappa}  2^{\ell(\lambda)} (-1)^{\col(\lambda/\mu)}  (-\beta/2)^{|\lambda/\mu| }  \gp_{\kappa}(x) \gp_{\mu/ \kappa}(y)
\]
where the outer sum is over all strict partitions $\mu$ with $\mu \subseteq \lambda$ and $\ell(\mu) = \ell(\lambda)$ such that $\SD_{\lambda/\mu}$ is a vertical strip and the inner sum is over strict partitions $\kappa$ with $\kappa \subseteq \mu$ (since $\gp_{\mu/\kappa} =0$ if $\kappa\not\subseteq \mu$).
Alternatively, using Corollary~\ref{to-prove2} to expand the right side of \eqref{skew-eq3} gives
\[
\gq_\lambda(x,y)  = \sum_{\eta} \sum_{\kappa}  2^{\ell(\kappa)}  (-1)^{\col(\eta/\kappa)}  (-\beta/2)^{|\eta/\kappa| }  \gp_{\kappa}(x) \gq_{\lambda / \eta}(y)
\]
where the outer sum is over all strict partitions $\eta$ with $\eta\subseteq \lambda$
(since $\gq_{\lambda/\eta} =0$ when $\eta \not\subseteq \lambda$)
and the inner sum over $\kappa$ with $\kappa\subseteq \eta$ and $\ell(\kappa) = \ell(\eta)$ such that $\SD_{\eta/\kappa}$ is a vertical strip.
Equating coefficients of $\gp_\kappa$  gives
\[
\sum_\eta 2^{\ell(\kappa)}
(-1)^{\col(\eta/\kappa)}  (\tfrac{-\beta}{2})^{|\eta/\kappa| }   \gq_{\lambda / \eta}
=
\sum_\mu
 2^{\ell(\lambda)} (-1)^{\col(\lambda/\mu)}  (\tfrac{-\beta}{2})^{|\lambda/\mu| } \gp_{\mu/ \kappa}
\]
where the sums are over certain strict partitions $\eta$ and $\mu$;
we can rewrite this as 
\[
\gq_{\lambda / \kappa}
=
\sum_{\mu}
 2^{\ell(\lambda)-\ell(\kappa)} (-1)^{\col(\lambda/\mu)}  (\tfrac{-\beta}{2})^{|\lambda/\mu| } \gp_{\mu/ \kappa}
 -
 \sum_{\eta}
(-1)^{\col(\eta/\kappa)} (\tfrac{-\beta}{2})^{|\eta/\kappa| }   \gq_{\lambda / \eta}
\]
where the first sum is over all strict partitions $\mu$ with $\kappa \subseteq \mu \subseteq \lambda$ and $\ell(\mu) = \ell(\lambda)$ such that $\SD_{\lambda/\mu}$ is a vertical strip, and the second sum is over all strict partitions $\eta$ with $\kappa\subsetneq \eta \subseteq \lambda$ and $\ell(\kappa) = \ell(\eta)$
such that $\SD_{\eta/\kappa}$ is a vertical strip.

If $\kappa = \lambda$ then $\gq_{\lambda/\kappa} = 1=\gp_{\lambda/\kappa}$ as the theorem predicts.
Otherwise, we may assume by induction that the desired formula holds for   $\gq_{\lambda/ \eta}$ with $\kappa\subsetneq \eta \subseteq \lambda$.
Substituting these formulas into the equation above
expands
%
$\gq_{\lambda/\kappa}$ as a linear combination of $\gp_{\mu/\nu}$'s where $\mu$ and $\nu$ range over the strict partitions with $\kappa \subseteq \nu \subseteq \mu \subseteq \lambda$ and $\ell(\kappa) = \ell(\nu) \leq \ell(\mu) = \ell(\lambda)$ such that $\SD_{\lambda/\mu}$ is a vertical strip. 
The coefficient of $\gp_{\mu/\kappa}$ in this expansion is $2^{\ell(\lambda)-\ell(\kappa)}(-1)^{\col(\lambda/\mu)}(\tfrac{-\beta}{2})^{|\lambda/\mu|}$ as desired. The coefficient of $\gp_{\mu/\nu}$ when $\kappa \subsetneq \nu\subseteq \mu \subseteq \lambda $ 
and  $\ell(\kappa)=\ell(\nu)\leq \ell(\mu) = \ell(\lambda)$ and $\SD_{\lambda/\mu}$ is a vertical strip is 
the sum
\[ 
-\sum_{\eta} (-1)^{\col(\eta/\kappa)}(\tfrac{-\beta}{2})^{|\eta/\kappa|} \( 2^{\ell(\lambda) - \ell(\eta)+\overlap(\nu/\eta)}  (-1)^{\col(\lambda/\mu)}  (\tfrac{-\beta}{2})^{|\nu/\eta| + |\lambda/\mu|}\)
\]  over all strict partitions $\eta$ where $\kappa \subsetneq \eta \subseteq \nu$ and $\SD_{\eta/\kappa}$ is a vertical strip.
Rewriting this as
$2^{\ell(\lambda)-\ell(\eta)}   (-1)^{\col(\lambda/\mu)}   (\tfrac{-\beta}{2})^{|\lambda/\mu|+|\nu/\kappa|}  \sum_{\eta } (-1)^{\col(\eta/\kappa)+1}  2^{\overlap(\nu/\eta)}$, we see that it suffices to show that 
 $\sum_{\eta} (-1)^{\col(\eta/\kappa)+1}  2^{\overlap(\nu/\eta)} = 2^{\overlap(\nu/\kappa)}$
 where the sum is again over $\eta$  with $\kappa \subsetneq \eta \subseteq \nu$ such that $\SD_{\eta/\kappa}$ is a vertical strip.
 After rearranging terms, this identity is the second half of
  Lemma~\ref{overlap_cancel}.
\end{proof}

By putting everything together we can also prove a generalization of \eqref{cauchy-eq}. 
This gives two shifted analogues of Yeliussizov's skew Cauchy identity for \emph{symmetric Grothendieck polynomials} \cite[Thm. 5.1]{Y2019}.

\begin{theorem}\label{GP-cauchy-thm}
Let $\mu$ and $\nu$ be strict partitions. Then
\be\label{GP-cauchy}
\sum_{\lambda} \GP_{\lambda \ss \mu}( x) \gq_{\lambda / \nu}( y) = \prod_{i,j\geq 1} \frac{1-\overline{x_i} y_j}{1-x_iy_j} \sum_\kappa \GP_{\nu \ss \kappa}( x) \gq_{\mu / \kappa}( y)
\ee
and
\be\label{GQ-cauchy}
\sum_{\lambda} \GQ_{\lambda \ss \mu}( x) \gp_{\lambda / \nu}( y) = \prod_{i,j\geq 1} \frac{1-\overline{x_i} y_j}{1-x_iy_j} \sum_\kappa \GQ_{\nu \ss \kappa}( x) \gp_{\mu / \kappa}( y)
\ee
where $\overline{x_i} := \frac{-x_i}{1+\beta x_i}$ and where $\lambda$ and $\kappa$ range over all strict partitions.\footnote{
Equivalently, one can restrict $\lambda$ to the strict partitions containing both $\mu$ and $\nu$ 
since otherwise $\GQ_{\lambda \ss \mu}( x) \gp_{\lambda / \nu}( y) =0$, and 
one can restrict $\kappa$ to the strict partitions contained in both $\mu$ and $\nu$
since otherwise $\GQ_{\nu \ss \kappa}( x) \gp_{\mu / \kappa}=0$.}
\end{theorem}

\begin{proof}
Let $w_1,w_2,w_3,\dots$ be a fourth set of commuting variables.
 The first expression in \eqref{GP-cauchy} is the coefficient of $\GP_{\mu}(w)\gq_{\nu}(z)$ 
 in 
\[ \sum_{\lambda} \GP_\lambda(w,x) \gq_\lambda(z,y) = \Delta(w;z)\Delta(w;y)\Delta(x;z)\Delta(x;y)\]
by \eqref{skew-eq2}, \eqref{cauchy-eq}, and \eqref{skew-eq3}.
Since $\Delta(x;y) = \prod_{i,j\geq 1} \frac{1-\overline{x_i} y_j}{1-x_iy_j}$,  to prove \eqref{GP-cauchy} it suffices to show that 
$\sum_\kappa \GP_{\nu \ss \kappa}( x) \gq_{\mu / \kappa}( y)$ is the coefficient of $\GP_{\mu}(w)\gq_{\nu}(z)$  in 
$\Delta(w;z)\Delta(w;y)\Delta(x;z)$. By \eqref{cauchy-eq} this product is equal to 
\[
\sum_{\kappa}\sum_\gamma\sum_{\lambda} \GP_\kappa(w)  \gq_\kappa(z)\cdot  \GP_\gamma(w) \gq_\gamma(y)\cdot  \GP_\lambda(x) \gq_\lambda(z).
\] 
Once we rearrange the terms of this expression as 
\[
\sum_{\kappa}\sum_\gamma\sum_{\lambda} \GP_\kappa(w)  \GP_\gamma(w) \cdot  \GP_\lambda(x)   \gq_\gamma(y)
 \cdot  \gq_\kappa(z)\gq_\lambda(z) \] 
it follows by \eqref{dual-exp-eq2} and \eqref{skew-exp-eq2} that it is equal to
\[
\sum_{\kappa}\sum_\gamma\sum_{\lambda}\sum_\mu\sum_\nu  \beta^{|\mu|-|\kappa|-|\gamma|}   a_{\kappa\gamma}^\mu\GP_\mu(w)  \cdot  \GP_\lambda(x)    \gq_\gamma(y)  \cdot \beta^{|\kappa|+ |\lambda|- |\nu|}  \widehat b_{\kappa\lambda}^{\nu}\gq_\nu(z)   .\]
On rearranging the terms of this expression to be 
\[
\sum_{\kappa}\sum_\gamma\sum_{\lambda}\sum_\mu\sum_\nu
   \GP_\mu(w)  \cdot   \beta^{|\kappa|+ |\lambda|- |\nu|}  \widehat b_{\kappa\lambda}^{\nu} \GP_\lambda(x) \cdot   \beta^{|\mu|-|\kappa|-|\gamma|} a_{\kappa\gamma}^\mu  \gq_\gamma(y)\cdot   \gq_\nu(z)
\]
we deduce by \eqref{skew-exp-eq} and \eqref{skew-g-def} that 
\[
\Delta(w;z)\Delta(w;y)\Delta(x;z) = \sum_{\kappa} \sum_\mu\sum_\nu  \GP_\mu(w) \cdot \GP_{\nu\ss\kappa}(x)\cdot \gq_{\mu/\kappa}(y) \cdot  \gq_\nu(z) .
\]
The coefficient of $\GP_{\mu}(w)\gq_{\nu}(z)$ in the last expression is 
   $\sum_\kappa \GP_{\nu \ss \kappa}( x) \gq_{\mu / \kappa}( y)$,
 so   \eqref{GP-cauchy} holds. The second identity \eqref{GQ-cauchy} follows by the same argument after
interchanging the symbols $\GP \leftrightarrow \GQ$, $\gp \leftrightarrow \gq$, $a \leftrightarrow b$ and $\widehat a \leftrightarrow \widehat b$.
\end{proof}

\section{Conjectural generating functions}\label{last-sect}

In this final section we discuss some conjectural formulas for 
$\gp_{\lambda/\mu}$ and $\gq_{\lambda/\mu}$
and two related dual functions.

Let $\mu \subseteq \lambda$ be strict partitions.
A \emph{shifted reverse plane partition} of shape $\lambda/\mu$
is
a filling 
of $\SD_{\lambda/\mu}$ by positive half-integers $ \{1'<1<2'<2<\dots\}$ such that rows and columns are weakly increasing. Examples include 
\be\label{rpp-ex}
\ytableausetup{boxsize=0.5cm,aligntableaux=center}
\begin{ytableau}
\none &\none & \none & 2  \\
\none &\none & 1' & 2 \\
\none &\none[\cdot] & 1' & 1 \\
\none[\cdot]  & \none[\cdot]  & \none[\cdot] & \none[\cdot] & 4' 
\end{ytableau}
\quand
\begin{ytableau}
\none &\none & \none & 2'  \\
\none &\none & 1' & 1 \\
\none &\none[\cdot] & 1' & 1 \\
\none[\cdot]  & \none[\cdot]  & \none[\cdot] & \none[\cdot] & 4 
\end{ytableau}
\ee
which both have shape $(5,3,2,1)/(4,1)$.
The \emph{weight} of a shifted reverse plane partition $T$ is the  monomial
$x^{\wt_\RPP(T)} :=  \prod_{i\geq 1} x_i^{c_i+r_i}$ 
where
 $c_i$ is the number of distinct columns of $T$ containing $i$
and $r_i$ is the number of distinct rows of $T$ containing  $i'$.
This monomial 
has degree $|\wt_\RPP(T)| := \sum_{i\geq 1} (c_i+r_i)$, which
may be less than $|T| := |\lambda/\mu|$.

Let $\MRPP_Q(\lambda/\mu)$ be the set of
shifted reverse plane partitions of shape $\lambda/\mu$.
Let 
$\MRPP_P(\lambda/\mu)$ 
denote the subset of  elements in $\MRPP_Q(\lambda/\mu)$
whose diagonal entries are all primed. The examples in \eqref{rpp-ex}
belong to $\MRPP_Q(5321/41)$ and $\MRPP_P(5321/41)$, respectively; both have weight $x_1^3 x_2 x_4$. 
Nakagawa and Naruse present 
the following conjecture in \cite{NakagawaNaruse}:

\begin{conjecture}[{\cite[Conj. 5.1]{NakagawaNaruse}}]
\label{conj1}
If $\mu\subseteq \lambda$ are strict partitions then
\[
\ba
 \gp_{\lambda/\mu} &= \sum_{T\in \MRPP_P(\lambda/\mu)} (-\beta)^{|\lambda/\mu| - |\wt_\RPP(T)|} x^{\wt_\RPP(T)} ,\\
\gq_{\lambda/\mu} &= \sum_{T\in \MRPP_Q(\lambda/\mu)} (-\beta)^{|\lambda/\mu| -  |\wt_\RPP(T)|} x^{\wt_\RPP(T)} .
\ea
\]
\end{conjecture}

\begin{example}\label{conj1-ex}
If $\lambda =(2,1)$ and $\mu=\emptyset$ then $ \MRPP_P(\lambda/\mu)$ consists of 
\[ 
\begin{ytableau}
\none[\cdot] &c' \\
a' & b 
\end{ytableau},
\quad
\begin{ytableau}
\none[\cdot] &c' \\
a' & b' 
\end{ytableau},
\quad
\begin{ytableau}
\none[\cdot] & b' \\
a' & a 
\end{ytableau},
\quad
\begin{ytableau}
\none[\cdot] &b' \\
a' & b'
\end{ytableau},
\quad
\begin{ytableau}
\none[\cdot] & a' \\
a' & a' 
\end{ytableau},
\quand
\begin{ytableau}
\none[\cdot] & b' \\
a' & a' 
\end{ytableau}
\]
for all positive integers $a<b<c$, so Conjecture~\ref{conj1} asserts that
\[
\ba
\gp_{21} &= 2\sum_{a<b<c} x_a x_b x_c + \sum_{a<b} (x_a^2x_b + x_ax_b^2) - \beta \sum_a x_a^2 - \beta \sum_{a<b} x_ax_b
 = s_{21} - \beta s_2.
\ea\]
If $\lambda =(2,1)$ and $\mu=\emptyset$ then adding primes to the diagonal is a weight-preserving 4-to-1 map $ \MRPP_Q(\lambda/\mu)\to \MRPP_P(\lambda/\mu)$
so Conjecture~\ref{conj1} also predicts that $\gq_{21} =4s_{21} - 4\beta s_2$.
These expression match the definitions from \eqref{cauchy-eq}.
\end{example}

\begin{remark}
The cited conjecture
\cite[Conj. 5.1]{NakagawaNaruse} only states these formulas when $\mu=\emptyset$ and $\beta=-1$.
However, this special case 
implies the general result.
In detail, if we knew the conjecture when $\beta=-1$ then we could derive the general statement 
using Proposition~\ref{recover-prop}.
In turn, if we knew
that 
$
 \gp_{\lambda} = \sum_{T\in \MRPP_P(\lambda)} (-\beta)^{|\lambda| - |\wt_\RPP(T)|} x^{\wt_\RPP(T)}$,
and hence that the combinatorial generating function were symmetric,
then we would have
 \[\ba
  \gp_\lambda(x,y) &=\sum_{\substack{\nu\subseteq\lambda \\ T\in \MRPP_P(\nu) \\ U\in \MRPP_P(\lambda/\nu)}}   (-\beta)^{|\lambda| - |\wt_\RPP(T)|- |\wt_\RPP(U)|} x^{\wt_\RPP(T)} y^{\wt_\RPP(U)}
  \\
  & = \sum_{\nu \subseteq \lambda} \gp_\nu(x)   \sum_{U\in \MRPP_P(\lambda/\nu)}   (-\beta)^{|\lambda/\nu| - |\wt_\RPP(U)|}  y^{\wt_\RPP(U)}.
\ea
  \]
We could then derive 
$
 \gp_{\lambda/\nu} = \sum_{T\in \MRPP_P(\lambda/\nu)} (-\beta)^{|\lambda/\nu| - |\wt_\RPP(T)|} x^{\wt_\RPP(T)}$
 by comparing the preceding identity with \eqref{skew-eq3} and equating the coefficients of $\gp_\nu$. 
 The reductions in $\gq$-case are similar.
\end{remark}

Nakagawa and Naruse give another conjectural formula for $\gq_\lambda$ as a certain Pfaffian in \cite[\S6.2]{NakagawaNaruse0},
but we will not discuss this here.

Write $\lambda^\top$ for the transpose of a partition $\lambda$.
Let $\omega$ be the operator acting on 
(possibly infinite) $\ZZ[\beta]$-linear combinations
of Schur functions $f=\sum_\lambda c_\lambda s_\lambda$
as $\omega(f) := \sum_\lambda c_\lambda s_{\lambda^\top}$.
For strict partitions $\mu$ and $ \lambda$ define
\be\label{jpq-eq} \jp_{\lambda/\mu} := \omega(\gp_{\lambda/\mu})
\quand
\jq_{\lambda/\mu} := \omega(\gq_{\lambda/\mu}),
\ee
setting $\jp_{\lambda} := \jp_{\lambda/\emptyset}$
and $\jq_{\lambda} := \jq_{\lambda/\emptyset}$.
These are finite $\ZZ[\beta]$-linear combinations of Schur functions
which are nonzero if and only if $\mu \subseteq \lambda$,
since the same is true of the $\gp$- and $\gq$-functions. When nonzero, $\jp_{\lambda/\mu}$ and $\jq_{\lambda/\mu}$ are homogeneous of degree $|\lambda|-|\mu|$ 
under the convention that $\deg(\beta) = \deg(x_i) = 1$.

As $\omega$ fixes all skew Schur $P$- and $Q$-functions \cite[Ex. 3(a), {\S}III.8]{Macdonald},
we have 
\be
P_{\lambda/\mu} = \jp_{\lambda/\mu}|_{\beta=0} 
\quand
Q_{\lambda/\mu} =\jq_{\lambda/\mu}|_{\beta=0}.
\ee
Additionally, since $\omega$ is an automorphism of the Hopf algebra 
of bounded degree symmetric power series in $\ZZ[\beta][[x_1,x_2,\dots]]$,
the identity \eqref{skew-eq3} is equivalent to
\be\label{j-eq}
\jp_\lambda(x,y) = \sum_\mu \jp_\mu(x) \jp_{\lambda/\mu}(y)
\quand
\jq_\lambda(x,y) = \sum_\mu \jq_\mu(x) \jq_{\lambda/\mu}(y)
\ee
where both sums are over all strict partitions $\mu$. 

We may also describe a conjectural 
generating function formula for $ \jp_{\lambda/\mu} $ and $ \jq_{\lambda/\mu} $.
This formula appears to be new.

A partition of a set $S$ is a set $\Pi$ of disjoint nonempty subsets $B \subseteq S$, called \emph{blocks},
with $S = \bigsqcup_{B \in \Pi} B$.
Assume $\mu \subseteq \lambda$.
We define a \emph{shifted bar tableau} of shape $\lambda/\mu$
to be a pair $T=(V,\Pi)$, where $V$ is a semistandard shifted tableau  of shape $\lambda/\mu$
and $\Pi$ is a partition of   $\SD_{\lambda/\mu}$ such that 
each block $B \in \Pi$ is a set of adjacent boxes containing the same entry in $V$.
Because $V$ is semistandard, each of these blocks must consist   
of a contiguous ``bar''
within a single row or single column. One might draw a shifted bar tableau as a picture like
\be\label{colors-eq}
\begin{young}[14pt][c] 
, & ]=![cyan!75]2  & =]![cyan!75]  \ynobottom & ![pink!75]3'\ynobottom  \\ 
]=![red!75]1 & =]![red!75]   & ]=]![blue!60] 1 & ]=]![pink!75] \ynotop & ![magenta!75]3 
\end{young}
\ee
to represent  
\[ T=(V,\Pi) =  \(\hs \begin{young}[14pt][c]
 , &  2 &  2 &  3'\\
1 & 1 &   1 &   3' &   3
\end{young}
,\hs \begin{young}[14pt][c] 
, & ]= \cdot & =] \cdot \ynobottom & \cdot\ynobottom  \\ 
]= \cdot & =]\cdot & ]=]\cdot & ]=] \cdot \ynotop &  \cdot
\end{young}\hs \).
\]

Let $\ValShYTQ(\lambda/\mu)$ denote the set of all shifted bar tableaux of shape $\lambda/\mu$
and let $\ValShYTP(\lambda/\mu)$
be the subset of such pairs $T=(V,\Pi)$
where $V $ has no primed diagonal entries.
Given $T = (V,\Pi) \in\ValShYTQ(\lambda/\mu)$ we set
\[ |T| := |\Pi|
\quand  x^T := \prod_{i\geq 1} x_i^{b_i} \]
where $b_i$ is the number of blocks in $\Pi$ containing $i$ or $i'$ in $V$.
The example $T$ shown in \eqref{colors-eq} 
belongs to $\ValShYTP(\lambda/\mu)$ for $\lambda = (5,3)$ and $\mu=\emptyset$
and has $|T| = 5$ and $x^T = x_1^2 x_2 x_3^2$.

If every block in $\Pi$ has size one then $|T| = |\lambda/\mu|$ and $x^T = x^V$.
If $\Pi$ has as few blocks as possible given $V$, then $x^T = x^{\wt_\RPP(V)}$.
Finally, observe that if $V$ is fixed, then the sum of $x^T$ over all $\Pi$
such that $T=(V,\Pi)$ is a shifted bar tableau is
$ \prod_{i \geq 1} x_i^{r_i + c_i} (x_i +1)^{m_i - r_i - c_i }$
where $r_i$ is the number of rows of $V$ containing an entry equal to $i$,
$c_i$ is the number of columns of $V$ containing an entry equal to $i'$,
and $m_i$ is the number of boxes of $V$ containing $i$ or $i'$.

\begin{conjecture}\label{conj2}
If $\mu\subseteq \lambda$ are strict partitions then
\[
\ba
 \jp_{\lambda/\mu} &= \sum_{T\in \ValShYTP(\lambda/\mu)} (-\beta)^{|\lambda/\mu|-|T|} x^{T} ,\\
\jq_{\lambda/\mu} &= \sum_{T\in \ValShYTQ(\lambda/\mu)} (-\beta)^{|\lambda/\mu| -|T|} x^{T} .
\ea
\]
\end{conjecture}

As with Conjecture~\ref{conj1}, to prove this result it would suffice  to assume $\mu=\emptyset$.

\begin{example}
If $\lambda =(2,1)$ and $\mu=\emptyset$ then $  \ValShYTP(\lambda/\mu)$ consists of 
\[ 
\begin{young}[14pt][c] 
, & ]=]![cyan!75]c    \\ 
]=]![red!75]a & ]=]![yellow!75]b    \end{young},
\quad
\begin{young}[14pt][c] 
, & ]=]![cyan!75]c    \\ 
]=]![red!75]a & ]=]![yellow!75]b'    \end{young},
\quad
\begin{young}[14pt][c] 
, & ]=]![cyan!75]b    \\ 
]=]![red!75]a & ]=]![yellow!75]a    \end{young},
\quad
\begin{young}[14pt][c] 
, & ]=]![cyan!75]b    \\ 
]=]![red!75]a & ]=]![yellow!75]b'    \end{young},
\quand
\begin{young}[14pt][c] 
, & ]=]![cyan!75]b    \\ 
]=![red!75]a & =]![red!75]    \end{young}
\]
for all positive integers $a<b<c$, so Conjecture~\ref{conj2} asserts that
\[
\ba
\jp_{21} &= 2\sum_{a<b<c} x_a x_b x_c + \sum_{a<b} (x_a^2x_b + x_ax_b^2) - \beta \sum_{a<b} x_ax_b
 = s_{21} - \beta s_{11}.
\ea\]
As in Example~\ref{conj1-ex}, there is a weight-preserving 4-to-1 map $ \ValShYTQ(21/\emptyset)\to \ValShYTP(21/\emptyset)$;
this is given by either removing all diagonal primes or applying 
 \[ 
 \left\{\ \begin{young}[14pt][c] 
, & ]=]![cyan!75]b    \\ 
]=![red!75]a & =]![red!75]    \end{young},\
\begin{young}[14pt][c] 
, & ]=]![cyan!75]b'    \\ 
]=![red!75]a & =]![red!75]    \end{young},\
 \begin{young}[14pt][c] 
, & ]=]![cyan!75]b'  \ynobottom  \\ 
]=]![red!75]a & ]=]![cyan!75] \ynotop    \end{young},\
 \begin{young}[14pt][c] 
, & ]=]![cyan!75]b'  \ynobottom  \\ 
]=]![red!75]a' & ]=]![cyan!75] \ynotop    \end{young}\ \right\}
\mapsto
\left\{\ \begin{young}[14pt][c] 
, & ]=]![cyan!75]b    \\ 
]=![red!75]a & =]![red!75]    \end{young}\ \right\}.\]
Thus Conjecture~\ref{conj2} also predicts that
$\jq_{21} =4s_{21} - 4\beta s_{11}$.
These formulas are consistent with Example~\ref{conj1-ex}
as $\jp_{21} = \omega(\gp_{21})$ and $\jq_{21} = \omega(\gq_{21})$.
\end{example}

\begin{remark}
One can systematically test Conjecture~\ref{conj1} by 
substituting into the Cauchy identity \eqref{cauchy-eq} 
both the set-valued generating functions for $\GP_\lambda$ and $\GQ_\lambda$
and the predicted reverse plane partition generating functions for $\gp_\lambda$ and $\gq_\lambda$,
then truncating all three expressions in \eqref{cauchy-eq} to finitely many variables and finite degree,
and finally checking that the resulting polynomials match. All computer calculations we have done along these lines support the conjecture.

To test Conjecture~\ref{conj2}, one can compute the Schur polynomial expansions of 
the right hand expressions in  Conjecture~\ref{conj1} and \ref{conj2} restricted to finitely many variables $x_1,x_2,\dots,x_n$.
If $n$ is large enough, then the corresponding expansions  should be related by transposing all partition indices.
We have used a computer to verify Conjecture~\ref{conj2} in this way  
for all strict partitions $\lambda$ with $|\lambda| \leq 6$.
\end{remark}

Conjectures~\ref{conj1} and \ref{conj2} are shifted analogues of results in \cite[\S9]{LamPyl}.
One may be able to adapt the operator methods used there and in \cite{Y2019} 
to prove both formulas. We will not pursue this here,
beyond verifying the one-row case:
 
 \begin{proposition}
 Conjectures~\ref{conj1} and \ref{conj2} hold when $\ell(\lambda) \leq 1$.
 \end{proposition}

 \begin{proof}
When $\ell(\lambda)=0$ the conjectures assert that $\gp_\emptyset=\gq_\emptyset=\jp_\emptyset=\jq_\emptyset  =1$.
This follows by comparing constant coefficients in \eqref{cauchy-eq}
 since $\GP_\emptyset = \GQ_\emptyset=1$.

Suppose  $\ell(\lambda)=1$.
We may assume $\mu=\emptyset$ in view of \eqref{skew-eq3} and \eqref{j-eq}.
 We may also assume $\beta=-1$
 since if $f $ is $ \gp_\lambda$, $\gq_\lambda$, $\jp_\lambda$, $\jq_\lambda$, or one of the conjectural generating functions,
 then by Proposition~\ref{recover-prop} we can recover $f$ from $f|_{\beta=-1}$ by 
substituting $x_i\mapsto -\beta^{-1} x_i$ for all $i$ and then multiplying the result by $(-\beta)^{|\lambda|}$.
 
Let
$ \overline \gp_\lambda := \sum_{T\in \MRPP_P(\lambda)} x^{\wt_\RPP(T)}$
and
$ \overline \gq_\lambda := \sum_{T\in \MRPP_Q(\lambda)}  x^{\wt_\RPP(T)}$.
We wish to show that $  \gp_{(n)}= \overline \gp_{(n)}$
and $  \gq_{(n)}= \overline \gq_{(n)}$ when $n$ is a positive integer.
Nakagawa and Naruse assert that these identities hold  \cite[Prop. 5.2]{NakagawaNaruse} but do not provide a proof.
Here is an argument. 
The function $\GP_\lambda(t)$ obtained by setting $x_1=t$ and $x_2=x_3=\dots=0$
is $t^n$ if $\lambda=(n)$ and $0$ if $\ell(\lambda)>1$.
Thus, if we set $x_1 = t$ and $x_i = 0$ for $i>0$
then 
\eqref{cauchy-eq} 
becomes 
\[
 \sum_{n\geq 0} t^n  \gq_{(n)}(y) = \prod_{j \geq 1}\tfrac{ 1 -\frac{-t}{1-t} y_j}{1-ty_j}
 .\]
Rewriting $y_j$ as $x_j$, we deduce that $\gq_{(n)}=\gq_{(n)}(x)$ is the coefficient of $t^n$ in
\[
\prod_{j \geq 1} \tfrac{ 1 -\frac{-t}{1-t} x_j}{1-tx_j}
 = \prod_{j\geq 1} (1 + x_j t+x_j t^2 +x_j t^3+ \dots)(1 + x_jt + x_j^2t^2 + x_j^3t^3+\dots).\]
 Each way of expanding this product into monomials corresponds to a unique one-row shifted reverse plane partition $T$:
 if we multiply the $a$th term of $(1 + x_j t+x_j t^2 +x_j t^3+ \dots)$
 with the $b$th term of $(1 + x_jt + x_j^2t^2 + x_j^3t^3+\dots)$, considering $1$ to be the $0$th term of both sums,
 then the resulting monomial is $x^{\wt_\RPP(T)} t^{|T|}$
 where 
$T$ is the one-row shifted reverse plane partition with exactly $a=a(j)$ entries equal to $j'$ and $b=b(j)$ entries equal to $j$ for each $j\geq 1$.
 It follows that the coefficient of $t^n$ in the above product is also $\overline \gq_{(n)}$,
 so
$\gq_{(n)}=\overline \gq_{(n)}$.
 
We have $\overline \gq_{(1)} = 2 \overline \gp_{(1)} = 2(x_1+x_2+ x_3+\dots)$,
and if $n\geq 2$ then $\overline \gp_{(n)}-\overline \gp_{(n-1)} = \sum_T x^{\wt_\RPP(T)}$
where the sum is over one-row shifted reverse plane partitions $T$ of size $n$ whose first two boxes contain distinct entries, the first of which is primed. Removing the prime from the first entry defines a weight-preserving bijection from the set of such $T$
to $\MRPP_Q(n) - \MRPP_P(n)$, so it follows that $\overline \gq_{(n)} = 2\overline \gp_{(n)}-\overline \gp_{(n-1)}$ when $n\geq 2$.
Since we likewise have $\gq_{(1)} = 2  \gp_{(1)}$ and 
$ \gq_{(n)} = 2 \gp_{(n)}- \gp_{(n-1)}$ when $n\geq 2$ by Corollary~\ref{to-prove2}, we deduce by induction that $\gp_{(n)}=\overline \gp_{(n)}$ for all $n$.
 
 An \emph{(unshifted) reverse plane partition} of shape $\lambda$ is a filling of the (unshifted) diagram
 $\D_\lambda := \{ (i,j) : 1 \leq j \leq \lambda_i\}$ by positive integers, such that rows and columns are weakly increasing.
 The weight of such an object is defined in the same way as for shifted reverse plane partitions.
As noted in \cite[Prop. 5.3]{NakagawaNaruse}, it is easy to see that 
$\overline \gp_{(n)} = \sum_T x^{\wt_\RPP(T)}$ where the sum is over all reverse plane partitions $T$
of any of the hook shapes $a1^{n-a} := (a,1,1,\dots,1)$ for $a \in [n]$; the relevant weight-preserving bijection
is given by moving all primed entries in an element of $ \MRPP_P(n)$ from the first row to the first  column with primes removed.
This sum is precisely $\sum_{a=1}^n g_{a1^{n-a}}$ where $g_\lambda$ is the \emph{dual stable Grothendieck polynomial}
discussed, for example, in \cite[\S9.1]{LamPyl}.

An \emph{(unshifted) bar tableau} of shape $\lambda$
is defined
in the same way as a shifted bar tableau, except the underlying tableau is a filling of $\D_\lambda$
(rather than $\SD_\lambda$) by positive integers (rather than positive half-integers); this is called a 
\emph{valued-set tableau} in \cite[\S9.8]{LamPyl}.
 The weight $x^T$ 
is the same as in the shifted case.
Let $\overline\jp_{\lambda} :=  \sum_{T\in \ValShYTP(\lambda)}  x^{T}$
and
$\overline\jq_{\lambda} :=  \sum_{T\in \ValShYTQ(\lambda)}  x^{T}$.
Then 
$\overline \jp_{(n)} = \sum_T x^{T}$ where the sum is over all bar tableaux $T$
of any of the hook shapes $a1^{n-a}$ for $a \in [n]$;  the relevant weight-preserving bijection
is again given by moving all primed entries in an element of $\ValShYTP(n)$ from the first row to the first  column with primes removed.
(Each primed entry comprised its own block in the first row and is assigned to its own block in the first column.)
This sum is precisely $\sum_{a=1}^n j_{a1^{n-a}}$ where $j_\lambda$ is the generating function
discussed in \cite[\S9.8]{LamPyl}, which has $j_\lambda = \omega(g_\lambda)$ for all partitions $\lambda$ \cite[Prop. 9.25]{LamPyl}.

Combining the last two paragraphs shows that $\overline \jp_{(n)} = \omega(\overline \gp_{(n)})$
so $\jp_{(n)} =  \omega( \gp_{(n)}) =  \omega(\overline \gp_{(n)}) = \overline \jp_{(n)} $ for all $n$.
Finally,
we have
$\overline\jq_{(1)} = 2 \overline \jp_{(1)}=2(x_1+x_2+ x_3+\dots)$, and if $n\geq 2$ then
$\overline\jp_{(n)}- \overline\jp_{(n-1)} = \sum_T x^T$
where the sum is over one-row shifted bar tableaux $T$ of size $n$ whose first two entries are unprimed but not in the same block. 
Adding a prime to the diagonal entry gives a weight-preserving bijection  from the set of such $T$
to $\ValShYTQ(n) - \ValShYTP(n)$, so 
 $\overline\jq_{(n)} =  2\overline \jp_{(n)} - \overline\jp_{(n-1)}$ when $n\geq 2$.
Since the same formulas relate $\jq_{(n)}$ to $\jp_{(n)}$ by the linearity of $\omega$,
we must have $\jq_{(n)} = \overline\jq_{(n)}$.
 \end{proof}
 
  Conjecture~\ref{conj2} has some consequences regarding the numbers $a^\lambda_{\mu\nu}$ and $b^\lambda_{\mu\nu}$.

\begin{theorem}\label{conj-lem}
Suppose the formula for $\jq_{\lambda/\mu}$ (respectively, $\jp_{\lambda/\mu}$) in Conjecture~\ref{conj2} holds.
Then the coefficients in the product expansions \eqref{skew-exp-eq2} satisfy $a^\lambda_{\mu\nu} =0$
(respectively, $b^\lambda_{\mu\nu} =0$) whenever $\ell(\lambda) > \ell(\mu) + \ell(\nu)$.
\end{theorem}

\begin{proof}
Our argument 
is based on Yeliussizov's proof of \cite[Thm. 8.4]{Y2019}.
Since consecutive diagonal entries in a semistandard shifted tableau differ by at least one,
the assumed formula in Conjecture~\ref{conj2}
implies that $\jq_{\nu}(x_1,\dots,x_n) \neq 0$ if and only if $\ell(\nu) \leq n$
and that $\jq_{\lambda/\mu}(x_1,\dots,x_n) = 0$  whenever $\ell(\lambda)  > n + \ell(\mu)$.
As  $\jq_{\nu}|_{\beta=0} = Q_{\nu} $,
the set of polynomials $\{ \jq_\kappa(x_1,\dots,x_n) : \ell(\kappa) \leq n\}$
is linearly independent over $\ZZ[\beta]$ since $\{ Q_\kappa(x_1,\dots,x_n) : \ell(\kappa) \leq n\}$ is linearly independent.

Applying $\omega$ to \eqref{skew-g-def} gives $\jq_{\lambda/\mu} = \sum_\kappa a_{\mu\kappa}^\lambda \beta^{|\lambda| - |\mu| - |\kappa|} \jq_\kappa$.
Thus if $a_{\mu\nu}^\lambda \neq 0$  and $n = \ell(\nu)$ then
$
\jq_{\lambda/\mu}(x_1,\dots,x_n) = \sum_{\ell(\kappa)\leq n} a_{\mu\kappa}^\lambda \beta^{|\lambda| - |\mu| - |\kappa|} \jq_\kappa(x_1,\dots,x_n) \neq 0
$.
But this means that $\ell(\lambda) \leq n + \ell(\mu) = \ell(\mu) + \ell(\nu)$ as desired.
 The claim about $b^\lambda_{\mu\nu}$
 follows by the same argument after swapping $\jq\leftrightarrow \jp$ and $Q\leftrightarrow P$.
\end{proof}

Ikeda and Naruse conjectured that the  product expansions in \eqref{skew-exp-eq2}
both have finitely many nonzero terms \cite[Conj. 3.1 and 3.2]{IkedaNaruse}.
For the $\GP$-functions, this follows from results in \cite{CTY},
which also establish that each coefficient $a^\lambda_{\mu\nu}  \in \NN $;
for other proofs see 
\cite[\S4]{Hamaker}, 
\cite[\S1.2]{M2021}, or \cite[\S8]{PechenikYong}.

 The same claim for the $\GQ$-functions appears still to be open,
but would be a consequence of Conjecture~\ref{conj2} by the following corollary.
 Specifically, this corollary shows that Conjecture~\ref{conj2} implies \cite[Conj. 3.2]{IkedaNaruse}.
Even given these conjectures, it is still an open problem to 
find a Littlewood--Richardson rule to compute $b^\lambda_{\mu\nu}$
outside the Pieri formula case $\nu=(p)$ handled in \cite{BuchRavikumar}.

\begin{corollary}\label{last-cor}
Suppose the formula for $\jp_{\lambda/\mu}$ in Conjecture~\ref{conj2} holds.
Then each product $\GQ_\mu\GQ_\nu$ is a finite $\ZZ[\beta]$-linear combination of $\GQ_\lambda$'s,
so the set $\{\GQ_\lambda : \text{$\lambda$ is a strict partition}\}$
is a $\ZZ[\beta]$-basis for a subring of $\ZZ[\beta][[x_1,x_2,\dots]]$.
\end{corollary}

\begin{proof} 
Theorem~\ref{conj-lem} implies that 
 $\GQ_\mu\GQ_\nu = \sum_{\ell(\lambda) \leq n} b^{\lambda}_{\mu\nu}\beta^{|\lambda| - |\mu| - |\nu|} \GQ_\lambda$
 when $n = \ell(\mu)+ \ell(\nu)$.
This must be a finite sum since the set 
 $\{ \GQ_\lambda(x_1,\dots,x_n) : \ell(\lambda) \leq n\}$
is a $\ZZ[\beta]$-basis for a subring of $\ZZ[\beta][x_1,\dots,x_n]$ by \cite[Prop. 3.2]{IkedaNaruse}.
 \end{proof}

For strict partitions $\mu\subseteq \lambda$
we may likewise define
\be\label{JPQ-eq}
\JP_{\lambda/\mu} := \omega(\GP_{\lambda/\mu})
\quand
\JQ_{\lambda/\mu} := \omega(\GQ_{\lambda/\mu}) ,
\ee
setting $\JP_{\lambda} := \JP_{\lambda/\emptyset}$
and
$\JQ_{\lambda} := \JQ_{\lambda/\emptyset}$.
Interpreting these as combinatorial generating functions is easier than in the dual case.
By \cite[Cor. 6.6]{LM2021} we  have
\be\label{J-eq}
\JP_{\lambda/\mu} = \GP_{\lambda/\mu}(\tfrac{x}{1-\beta x})
\quand
\JQ_{\lambda/\mu} = \GQ_{\lambda/\mu}(\tfrac{x}{1-\beta x})
\ee
where 
$f(\frac{x}{1-\beta x})$ denotes the power series obtained
from $f \in \ZZ[\beta][[x_1,x_2,\dots]]$
 by substituting $x_i \mapsto \frac{x_i}{1-\beta x_i} = x_i + \beta x_i^2 + \beta^2 x_i^3 + \dots$ for all $i$.

Using \eqref{J-eq} it is straightforward to turn the formulas \eqref{skew-GP-GQ-def} into expressions for $\JP_{\lambda/\mu}$ and $\JQ_{\lambda/\mu}$ as generating functions 
 $\sum_T \beta^{|T| - |\lambda/\mu|} x^T$
 for \emph{semistandard weak set-valued shifted tableaux}\footnote{
A \emph{semistandard weak set-valued shifted tableau} of shape $\lambda/\mu$ is 
defined in the same way as a semistandard set-valued shifted tableau, except the
entries of such a tableau are 
finite nonempty \emph{multisets} of positive half-integers.} of shape $\lambda/\mu$, with primed numbers excluded from the diagonal in the $\JP$ case. See \cite[\S3]{Hamaker}, for example, where the power series denoted
$K_\lambda$ is the same as $\JP_{\lambda}|_{\beta=1}$.
Since $\omega$ is  linear, if we define
\be
\JP_{\lambda\ss\mu} := \omega(\GP_{\lambda\ss\mu})
\quand
\JQ_{\lambda\ss\mu} := \omega(\GQ_{\lambda\ss\mu}) 
\ee
then 
it also holds that 
\be\label{JJ-eq}
\JP_{\lambda\ss\mu} = \GP_{\lambda\ss\mu}(\tfrac{x}{1-\beta x})
\quand
\JQ_{\lambda\ss\mu} = \GQ_{\lambda\ss\mu}(\tfrac{x}{1-\beta x})
\ee and
 it follows from \eqref{skew-eq2} that
\be\label{skew-e4}
\ba
\JP_\lambda(x,y) &= \sum_{\nu } \JP_\nu(x) \JP_{\lambda\ss\nu}(y),
\\
\JQ_\lambda(x,y) &= \sum_{\nu} \JQ_\nu(x) \JQ_{\lambda\ss\nu}(y),
\ea
\ee
where the sums are over all strict partitions $\nu$.

One can write down Cauchy identities relating each pair of $\GP$/$\jq$, $\GQ$/$\jp$,
$\JP$/$\gq$, $\JQ$/$\gp$,
$\JP$/$\jq$, and $\JQ$/$\jp$ functions.
These 
provide a shifted analogue of \cite[Cor. 6.3]{Y2019}.
Recall that $\overline{x_i} := \frac{-x_i}{1+\beta x_i}
$ and $ \Delta(x,y) := \prod_{i,j\geq 1} \frac{1-\overline{x_i} y_j}{1-x_iy_j}.$

\begin{corollary}
Let $\mu$ and $\nu$ be strict partitions. 
Then:
\ben

\item[(a)] $\Delta(x,-y) \sum_{\lambda} \GP_{\lambda \ss \mu}( x) \jq_{\lambda / \nu}( y) = \sum_\kappa \GP_{\nu \ss \kappa}( x) \jq_{\mu / \kappa}( y)$;

\item[(b)] $\Delta(x,-y)\sum_{\lambda} \GQ_{\lambda \ss \mu}( x) \jp_{\lambda / \nu}( y) =\sum_\kappa \GQ_{\nu \ss \kappa}( x) \jp_{\mu / \kappa}( y)$;

\item[(c)] $\Delta(-x,y)\sum_{\lambda} \JP_{\lambda \ss \mu}( x) \gq_{\lambda / \nu}( y) = \sum_\kappa \JP_{\nu \ss \kappa}( x) \gq_{\mu / \kappa}( y)$;

\item[(d)] $\Delta(-x,y)\sum_{\lambda} \JQ_{\lambda \ss \mu}( x) \gp_{\lambda / \nu}( y) =\sum_\kappa \JQ_{\nu \ss \kappa}( x) \gp_{\mu / \kappa}( y)$;

\item[(e)] $\sum_{\lambda} \JP_{\lambda \ss \mu}( x) \jq_{\lambda / \nu}(y) = \Delta(-x,-y) \sum_\kappa \JP_{\nu \ss \kappa}( x) \jq_{\mu / \kappa}( y)$;

\item[(f)] $\sum_{\lambda} \JQ_{\lambda \ss \mu}( x) \jp_{\lambda / \nu}( y) =\Delta(-x,-y)  \sum_\kappa \JQ_{\nu \ss \kappa}( x) \jp_{\mu / \kappa}( y)$.

\een
As usual the sums are over all strict partitions $\lambda$ and $\kappa$.
\end{corollary}

\begin{proof}
As $\omega$ interchanges the elementary and complete symmetric functions $e_n := s_{1^n} $ and $h_n:=  s_{(n)} $,
it follows that $\omega$ (extended to an automorphism of the ring of symmetric functions in the $x_i$ variables
over $\ZZ[t]$) also interchanges
 $E(t;x) := \sum_{n\geq 0} e_n t^n = \prod_{i \geq 1} (1+x_i t)$ and $H(t;x) :=\sum_{n\geq 0} h_n t^n = \prod_{i\geq 1} \frac{1}{1-x_i t}$.
 The desired identities are straightforward to derive from Theorem~\ref{GP-cauchy} using this observation with
\eqref{jpq-eq} and \eqref{JJ-eq}, since one has
\[
\Delta(x,y) = \prod_{i \geq 1} E(-\overline{x_i}; y) H(x_i;y)
\quand 
-\overline{x_i}|_{x_i \mapsto \frac{x_i}{1-\beta x_i}} = 
\frac{\frac{x_i}{1-\beta x_i}}{1+\beta \frac{x_i}{1-\beta x_i}}
= x_i.
\]
For example, substituting $x_i \mapsto \frac{x_i}{1-\beta x_i}$ and then
applying the version of $\omega$ which acts only on symmetric functions in the $y$-variables  
transforms Theorem~\ref{GP-cauchy} to
\[
\sum_{\lambda} \JP_{\lambda \ss \mu}( x) \jq_{\lambda / \nu}(y) = \prod_{i \geq 1} H(x_i; y) E(\tfrac{x_i}{1-\beta x_i};y)\sum_\kappa \JP_{\nu \ss \kappa}( x) \jq_{\mu / \kappa}( y).
\]
This implies (e) as 
$ \prod_{i \geq 1} H(x_i; y) E(\tfrac{x_i}{1-\beta x_i};y) = \prod_{i,j\geq 1} \frac{1 + \tfrac{x_iy_j}{1-\beta x_i} }{1-x_i y_j} = \Delta(-x,-y).$
The other parts are derived in a similar way.
\end{proof}

\end{document}